\documentclass[11pt]{amsart} 
\usepackage[vmargin=2cm,hmargin=2cm]{geometry}

\usepackage[latin1]{inputenc} 
\usepackage{enumerate} 
\usepackage{hyperref}
\usepackage{amsfonts,amsmath,amssymb}

\usepackage{pgf} 
\usepackage{tikz} \usepgfmodule{plot}
\usepgflibrary{plothandlers} 
\usetikzlibrary{shapes.geometric}
\usetikzlibrary{arrows}

\usepackage{graphicx}

\newtheorem{theorem}{Theorem} 
\newtheorem{lemma}[theorem]{Lemma}
\newtheorem{corollary}[theorem]{Corollary}
\newtheorem{proposition}[theorem]{Proposition}

\newtheorem{defn}[theorem]{Definition}

\theoremstyle{remark} 
\newtheorem{example}[theorem]{Example}
\newtheorem{remark}[theorem]{Remark}

\DeclareMathOperator{\ff }{F} 
\DeclareMathOperator{\D}{D}
\DeclareMathOperator{\E}{E}

\DeclareMathOperator{\Ap}{Ap}

\def\pe[#1]{{\left\lceil #1\right\rceil}} 
\def\peb[#1]{{\left\lfloor #1\right\rfloor}}

\newcommand{\cardinal}{\#} 

\newcommand\m{\mathsf{m}}
\def\oo{\mathsf{o}}

\newcommand\GAP{\textsf{GAP}}

\title[On the weight hierarchy of codes coming from semigroups with two generators]
{On the weight hierarchy of codes coming from semigroups with two generators}

\author{M. Delgado} 
\address{CMUP, Departamento de Matematica, Faculdade de
  Ciencias, Universidade do Porto, Rua do Campo Alegre 687, 4169-007 Porto,
  Portugal} 
\email{mdelgado@fc.up.pt} 
\thanks{The first author was partially
  funded by the European Regional Development Fund through the program
  COMPETE and by the Portuguese Government through the FCT - Funda{\c
    c}{\~a}o para a Ci{\^e}ncia e a Tecnologia under the project
  PEst-C/MAT/UI0144/2011. He benefited also of the sabbatical grant
  SFRH/BSAB/1156/2011.}

\author{J. I. Farr\'{a}n} 
\address{Departamento de Matem\'atica Aplicada,
  Escuela Universitaria de Inform\'atica, Campus de Segovia - Universidad de
  Valladolid, Plaza de Santa Eulalia 9 y 11 - 40005 Segovia, Spain}
\email{jifarran@eii.uva.es} 
\thanks{The second author is supported by the
  projects MICINN-MTM-2007-64704 and MTM2012-36917-C03-01.}

\author{P. A. Garc\'{\i}a-S\'{a}nchez} 
\address{Departamento de \'Algebra,
  Universidad de Granada, 18071 Granada, Espa\~na} 
\email{pedro@ugr.es}
 
\author{D. Llena} 
\address{Departamento de Matem\'aticas, Universidad de
  Almer\'{\i}a, 04120 Almer\'{\i}a, Espa\~na} 
\email{dllena@ual.es}

\thanks{The third and fourth authors are supported by the projects
  MTM2010-15595, FQM-343 and FEDER funds}

\thanks{The third author is also supported by the project FQM-5849.}

\date{\today}

\keywords{AG codes, weight hierarchy, numerical semigroups, order bounds,
  Goppa-like bounds, Feng-Rao numbers.}

\subjclass[2010]{11T71,20M14,11Y55}

\begin{document}

\begin{abstract}

The weight hierarchy of one-point algebraic geometry codes can be estimated by means of the generalized order bounds, which are described in terms of a certain Weierstrass semigroup. The asymptotical behaviour of such bounds for $r\geq 2$ differs from that of the classical Feng-Rao distance ($r=1$) by the so-called Feng-Rao numbers. This paper is addressed to compute the Feng-Rao numbers for numerical semigroups of embedding dimension two (with two generators), obtaining a closed simple formula for the general case by using numerical semigroup techniques. These involve the computation of the Apéry set with respect to an integer of the semigroups under consideration. The formula obtained is applied to lower-bounding the generalized Hamming weights, improving the bound given by Kirfel and Pellikaan in terms of the classical Feng-Rao distance. We also compare our bound with a modification of the Griesmer bound, improving this one in many cases. 
\end{abstract}

\maketitle


\section{Introduction}\label{sec:intro}

Kirfel and Pellikaan introduced in~\cite{KirPel} the concept of array of codes.  More generally, the concepts of order function and weight function allows us to define arrays of codes with the same features (see~\cite{HvLP}). For such an array there is a majority voting algorithm for decoding efficiently up to half the so-called Feng-Rao distance. This distance is obtained by numerical computations in a certain underlying numerical semigroup.

These results are actually a linear algebra formalization of the Feng-Rao decoding algorithm and the Feng-Rao distance $\delta_{RF}(m)$\/, introduced in~\cite{FR} for one-point algebraic geometry codes (AG codes in short).  The Feng-Rao distance (also known as {\em order bound} in the literature) becomes a lower bound for the minimum distance of the involved error-correcting codes.

In the case of one-point AG codes, the Feng-Rao distance improves the lower bound for the minimum distance given by the Riemann-Roch theorem, that is called the Goppa distance.  This result has a translation in~\cite{KirPel} to the case of arrays of codes, namely
\[
\delta_{FR}(m+1)\geq m+2-2g
\]
if $m>2g-2$, and the equality holds for $m>>0$. The number $m+2-2g$ corresponds to the Goppa bound.

Even though the Feng-Rao distance was introduced just for Weierstrass semigroups and with decoding purposes, it is just a combinatorial concept that makes sense for arbitrary numerical semigroups, so that it can be computed just with numerical semigroup techniques. The problem of computing Feng-Rao distances has been broadly studied in the literature for different types of numerical semigroups (see~\cite{WSPink},~\cite{Arf} or~\cite{KirPel}).

Later on, the concept of minimum distance for an error-correcting code has been generalized to the so-called {\em generalized Hamming weights} and the {\em weight hierarchy}\/.  These concepts were independently introduced by Helleseth et al. in~\cite{HKM} and Wei in~\cite{Wei} for applications in coding theory and cryptography, respectively.

On the other hand, the Feng-Rao distance has been generalized in a natural way to higher weights (see~\cite{HeiPel}).  The obtained generalized Feng-Rao distances (or {\em generalized order bounds}), defined on the underlying numerical semigroup for an array of codes (or a weight function, in a modern setting), become lower bounds for the corresponding generalized Hamming weights.  However, the computation of these generalized Feng-Rao distances is a much more complicated problem than in the classical case. This means that very few results are known about this subject, and they are completely scattered in the literature (see for example~\cite{AngCar},~\cite{numericalsgps},~\cite{fr-intervalos},~\cite{WCC} or~\cite{HeiPel}).

This paper is addressed to the asymptotical behaviour of the generalized Feng-Rao distances, that is, $\delta_{FR}^{r}(m)$ for $r\geq 2$ and $m>>0$. In fact, it was proved in~\cite{WCC} that
\[\delta_{FR}^{r}(m)=m+1-2g+E_{r}
\]
for $m>>0$ (details are made precise in the next section). The number $E_{r}\equiv \mathrm E(S,r)$ is called the $r$th Feng-Rao number of the semigroup $S$, and they are unknown but for very few semigroups and concrete $r$'s.  For example, it was proved in~\cite{JMDA} that
\begin{equation}\label{ErRhor}
  \textrm E(S,r)=\rho_{r}
\end{equation}
for hyper-elliptic semigroups $S=\langle 2,2g+1\rangle$, with multiplicity 2 and genus $g$, and for Hermitian-like semigroups $S=\langle a,a+1\rangle$, where $S=\{\rho_1=0<\rho_2<\cdots\}$.

In~\cite{fr-intervalos} the authors compute the Feng-Rao numbers for numerical semigroups generated by intervals, generalizing the techniques for the Hermitian-like case, but not obtaining the same formula (note that such semigroups are not symmetric in general). In the present paper we generalize the results of~\cite{JMDA}, obtaining the above formula \eqref{ErRhor} for the general case of embedding dimension two numerical semigroups, that is, $S=\langle a,b\rangle$ generated by two elements.  In particular, we get a lower bound for the generalized Hamming weights in an array of codes whose associated semigroup is such an $S$. This bound improves the one given in~\cite{KirPel} in terms of the classical Feng-Rao distance. In fact, once the Feng-Rao number $E_{r}$ is computed, we get a lower bound for the generalized Feng-Rao distance
\[
\delta_{FR}^{r}(m)\geq m+1-2g+E_{r}
\]
for $m\geq c$ (see \cite{WCC}).

The computation of $\delta_{FR}^r$ is related to divisors of multiple elements in a numerical semigroups, and we show that these can be calculated by using Ap\'ery sets.  These sets are a powerful tool in the study of numerical semigroups, basically because they provide fast computations, and they were known only when $n$ equals to one of the generators. We obtain a general expression for the Ap\'{e}ry sets of a semigroup $S$ with two generators, with respect to any integer $n$. 

The paper is organized as follows. Section~\ref{sec:backgound} sets the general definitions concerning numerical semigroups, Feng-Rao distances, Feng-Rao numbers and amenable sets. Computations on embedding dimension two numerical semigroups are introduced in main Section~\ref{sec:basic-tools}.  More precisely, we revise the calculations on grounds and triangles in~\cite{fr-intervalos} for the case of semigroups with two generators, obtaining the desired formula \eqref{ErRhor} for the Feng-Rao numbers in Theorem~\ref{main} by working with amenable sets.  Experimental results with the \GAP~\cite{GAP4} package \texttt{numericalsgps}~\cite{numericalsgps} pointed precisely to this formula for this type of semigroups, and were actually the starting point and motivation to write this paper. The paper ends with some examples and conclusions in Section~\ref{sec:examples_conclusions}.


\section{Feng-Rao distances on numerical semigroups}\label{sec:backgound} 
Let $S$ be a numerical semigroup, that is, a submonoid of $\mathbb N$ such that $\cardinal (\mathbb N\setminus S)<\infty$ and $0\in S$.  Denote by $g:=\cardinal (\mathbb N\setminus S)$ the \emph{genus} of $S$. Note that if $S$ is the Weierstrass semigroup of a curve $\chi$ at a point $P$, $g$ equals precisely to the geometric genus of $\chi$, and the elements of $\mathrm G(S):=\mathbb N\setminus S=\{\ell_{1}<\cdots<\ell_{g}\}$ are called the \emph{gaps} of $S$ (for the case $S$ being a Weierstrass semigroup, they are also known as Weierstrass gaps of $\chi$ at $P$).

Let $c\in S$ be the \emph{conductor} of $S$, that is the (unique) element in $S$ such that $c-1\notin S$ and $c+l\in S$ for all $l\in\mathbb N$.  We obviously have $c\leq 2g$, and thus the \lq\lq largest gap\rq\rq \ of $S$ is $\ell_{g}\doteq c-1\leq 2g-1$. The number $\ell_{g}$ is precisely the \emph{Frobenius number} of $S$, denoted by $\mathrm F(S)$ in the literature.  The semigroup $S$ is called \emph{symmetric} provided $r\in S$ if and only if $c-1-r\notin S$, for all $r\in\mathbb Z$.  This is equivalent to say that $c=2g$ or $\mathrm F(S)=2g-1$.

The \emph{multiplicity} of a numerical semigroup is the least positive integer belonging to it.  Note that if $S$ is the value semigroup of a curve $\chi$ at a point $P$, this number equals to the multiplicity of $\chi$ at the point $P$.

We say that a numerical semigroup $S$ is generated by a set of elements $G\subseteq S$ if every element $m\in S$ can be written as a linear combination
\[
m=\displaystyle\sum_{g\in G}\lambda_{g}g,
\]
where finitely many $\lambda_{g}\in\mathbb N$ are non-zero.  It is well-known that every numerical semigroup is finitely generated, that is, there exists a finite set $G$ that is a generator set for $S$.  Furthermore, every such generator set contains the set of irreducible elements, where $m\in S$ is \emph{irreducible} if $m=a+b$ and $a,b\in S$ implies $a\cdot b=0$.  The set of irreducibles actually generates $S$, so that it is usually called \lq\lq the\rq\rq \ \emph{generator set} of $S$ (and thus its elements are sometimes referred as \emph{generators}). The number of irreducibles is called \emph{embedding dimension} of $S$ (see~\cite{NS} for further details). The smallest generator is precisely the \emph{multiplicity}.

Finally, if we enumerate the elements of $S$ in increasing order
\[
S=\{\rho_1=0<\rho_2<\cdots\},
\]
we note that every $m\geq c$ is the $(m+1-g)$th element of $S$, that is $m=\rho_{m+1-g}\,$.

Following~\cite{NS}, for $a,b\in \mathbb Z$ given, we say that $a$ \emph{divides} $b$, and write \[a\le_S b, \hbox{ if }b-a\in S.\] This binary relation is an order relation.
 
The set $\D(a)$ denotes the set of \emph{divisors} of $a$ in $S$, and for a given $M=\{m_1,\ldots,m_r\}\subseteq S$, we write $\D(M)=\D(m_1,\ldots,m_r)=\displaystyle\bigcup_{i=1}^r \D(m_i)$.  Thus, from now on, we will use the term \emph{divisors} to refer to the elements in the sets $\D(\cdot)$.

Note that $\D(m_{1})\subseteq[0,m_{1}]$, and $s\in \D(m_{1})$ implies $\D(s)\subseteq \D(m_{1})$.  The following result was proved in~\cite{fr-intervalos}.

\begin{lemma}\label{lema:divisors}
  $\D(x)=S\cap (x-S)$.
\end{lemma}

\begin{remark}\label{orden-triangulos}
  As an immediate consequence we get $\cardinal (\D(m+\rho_n) \cap [m,\infty))=n$ for $m\ge c$.
\end{remark}

The above inclusion $\D(m)\subseteq \D(m+p)$, for all $p\in S$, is very useful for practical computations. Moreover, we easily get the following result (see~\cite{fr-intervalos}).

\begin{proposition}\label{prop:well-known}
  If $m\ge 2c-1$, then $\cardinal \D(m)=\cardinal \left(S\cap (m-S)\right)=m+1-2g$.
\end{proposition}

We now introduce the definitions of the generalized Feng-Rao distances and summarize known results about them.

\begin{defn}\label{FRdist}
  Let $S$ be a numerical semigroup. The (classical) \emph{Feng-Rao distance} of $S$ is defined by the function
  \[
  \begin{array}{rcl}
    \delta_{FR}\;:\;S&\longrightarrow&\mathbb N \\
    m&\mapsto&\delta_{FR}(m):=\min\{\cardinal \D(m_1)\;|\;m_1\geq m,\;\;m_1\in S\}.
  \end{array}
  \]
\end{defn}

There are some well-known facts about the $\delta_{FR}$ function for an arbitrary semigroup $S$ (see~\cite{HvLP},~\cite{KirPel} or~\cite{WSPink} for further details).  The most important one for our purposes is that $\delta_{FR}(m)\geq m+1-2g$ for all $m\in S$ with $m\geq c$, and that equality holds if moreover $m \geq 2c-1$.

The classical Feng-Rao distance corresponds to $r=1$ in the following definition.

\begin{defn}\label{FRgen}

  Let $S$ be a numerical semigroup. For any integer $r\geq 1$, the \emph{$r$th Feng-Rao distance} of $S$ is defined by the function
  \[
  \begin{array}{rcl}
    \delta_{FR}^{r}\;:\;S&\longrightarrow&\mathbb N \\
    m&\mapsto&\delta_{FR}^{r}(m):=
    \min\{\cardinal \D(m_{1},\ldots,m_{r})\;|\;m\leq m_{1}<\cdots<m_{r},\;\;m_{i}\in S\}.
  \end{array}
  \]
\end{defn}

Very few results are known for the numbers $\delta_{FR}^{r}$, and their computation is very hard from both a theoretical and computational point of view. The main result we need describes the asymptotical behaviour for $m>>0$, and was proved in~\cite{WCC}. As we already mentioned in the introduction, this result tells us that there exists a certain constant $E_{r}=\mathrm E(S,r)$, depending on $r$ and $S$, such that
\begin{equation}\label{Er}
  \delta_{FR}^{r}(m)=m+1-2g+E_{r}
\end{equation}
for $m \geq 2c-1$.  This constant is called the \emph{$r$th Feng-Rao number} of the semigroup $S$. Furthermore, it is also true that $\delta_{FR}^{r}(m)\geq m+1-2g+\mathrm E(S,r)$ for $m\geq c$, and equality holds if $S$ is symmetric and $m=2g-1+\rho$ for some $\rho\in S\setminus\{0\}$ (see~\cite{WCC}). Somehow, this constant measures the difference between $\delta_{FR}^{r}(m)$ and $\delta_{FR}(m)$ for sufficiently large $m$, being $\mathrm E(S,1)=0$. For the trivial semigroup with $g=0$, it is easy to check that $\mathrm E(S,r)=r-1$.

We summarize some general properties of the Feng-Rao numbers, for $r\geq 2$ and $S$ fixed, with $g\geq 1$  (see again~\cite{WCC} for the details):

\begin{enumerate}[1.]

\item The function $\mathrm E(S,r)$ is non-decreasing in $r$.

\item $r\leq \mathrm E(S,r)\leq\rho_{r}\,$.  If furthermore $r\geq c$, then $\mathrm E(S,r)=\rho_{r}=r+g-1$.

\end{enumerate}

The computation of the Feng-Rao numbers is a very hard task, even in very simple examples.  So far, only the second Feng-Rao number ($r=2$) is computed in the literature, with either a general algorithm based on Ap\'{e}ry systems, or concrete formulas for simple examples by counting deserts (see~\cite{WCC}).  More precisely, the only exact formula, given in~\cite{WCC}, is that \[
\mathrm E(S,2)=\rho_{2}
\]
for $S$ generated by two elements.

In a previous work~\cite{JMDA}, and by using different techniques, we have found two families of numerical semigroups with $\mathrm E(S,r)=\rho_r$: that of numerical semigroups with multiplicity two (hyper-elliptic), and those embedding dimension two numerical semigroups generated by a positive integer $a$ and $a+1$ (hermitian-like).  In this paper we generalize this result to the whole family of embedding dimension two numerical semigroups.

Note that in general this bound is not attained for other kinds of semigroups, not even for $r=2$.  For example, if we consider the semigroup $S=\langle 6,13,14,15,16,17\rangle$ then $\mathrm E(S,2)=3<\rho_{2}=6$.

The following definitions are addressed to find the minimum required by the definition of Feng-Rao distance.

\begin{defn}\label{def:optimal_configuration}
  Let $S$ be a numerical semigroup and let $m\in S$.  A finite subset of $S\cap [m,\infty)$ is called a $(S,m)$-\emph{configuration}, or simply a configuration.  A configuration $M$ of cardinality $r$ is said to be \emph{optimal} if $\delta_{FR}^r(m)=\cardinal \D(M)$.
\end{defn}

Motivated by Formula (\ref{Er}) and Proposition~\ref{prop:well-known}, in the sequel we denote by $\m$ any integer greater than or equal to $2c-1$, where $c$ is the conductor of the semigroup under consideration.
\begin{defn}
  Let $S$ be a numerical semigroup with conductor $c$.  Let $M = \{m_1,\ldots,m_r\}\subseteq S$ with $\m = m_1<\cdots<m_r$.  We say that the set $M$ is $(S,\m,r)$-\emph{amenable} if:
  \begin{equation}\label{eq:def_amenable}
    \mbox{for all }i\in\{1,\ldots,r\},\D(m_i)\cap [\m ,\infty)\subseteq M. 
  \end{equation}
\end{defn}

We will refer a set satisfying (\ref{eq:def_amenable}) as being $\m$-\emph{closed under division}.  So, a subset of $S\cap[\m ,\infty)$ with cardinality $r$ is $(S,\m,r)$-amenable if and only if it contains $\m$ and is $\m$-closed under division.

When no confusion arises or only the concept is important, we simply say \emph{amenable} set.

The following is immediate from the definition (it has also been stated in~\cite[Lemma~40]{fr-intervalos}).

\begin{lemma}\label{amenable-menos-maximo}
  Let $M\neq \{\m\}$ be an amenable set. Then $M\setminus \{\max (M)\}$ is again an amenable set.
\end{lemma}

The importance of amenable sets comes from the following result, which states that among the optimal configurations of cardinality $r$ there is at least one $(S,\m,r)$-amenable set.

\begin{proposition}\cite[Proposition 10]{fr-intervalos}\label{condiciones-ms}
Let $S$ be a numerical semigroup with conductor $c$ and let $\m\ge 2c-1$. Let $r$ be a positive integer.  Among the $(S,\m)$-optimal configurations of cardinality $r$ there is one $(S,\m,r)$-amenable set.
\end{proposition}

The following lemma is of extreme importance, since it will allow us to concentrate in computing amenable sets whose so-called grounds have as few divisors as possible.

\begin{lemma}\label{sombra}
  Let $S$ be a numerical semigroup minimally generated by $\{n_1<\cdots<n_e\}$ with conductor $c$.  Let $\m \geq 2c-1$ and let $M=\{\m =m_1<\cdots<m_r\}$ be an amenable set.  Define $L=M\cap [\m,\m+n_e)$. Then $L$ is an amenable set,
  \[\D(L)\cap [0,\m)=\D(M)\cap[0,\m)\]
  and
  \[\cardinal \D(M)=\cardinal (\D(L)\cap [0,\m))+\cardinal M.\]
\end{lemma}
\begin{proof}
  Clearly $L$ is amenable.

  By~\cite[Lemma 13]{fr-intervalos}, $\D(M)=(M\setminus L)\cup\D(L)$. Hence $\D(M)\cap[0,\m)=((M\setminus L)\cup\D(L))\cap[0,\m)=\D(L)\cap[0,\m)$.

  Also $\D(M)=(M\setminus L)\cup (\D(L)\cap[\m,\infty)) \cup (\D(L)\cap [0,\m))$. As $L$ is amenable, $\D(L)\cap[\m,\infty)=L$. Hence $\D(M)= (M\setminus L)\cup L\cup (\D(L)\cap [0,\m)) =M\cup (\D(L)\cap[0,\m))$, and this union is a disjoint union, whence $\cardinal \D(M)=\cardinal (\D(L)\cap [0,\m))+\cardinal M$.
\end{proof}

Let $S$ be a numerical semigroup, $n$ an integer and consider the following set:
\[\Ap(S,n)= \{s\in S\mid s-n\not\in S\}.\]
This set is said to be the \emph{Ap\'ery set of $S$ with respect to $n$}. The Ap\'ery set with respect to an element of $S$ is one of the most powerful ingredients in the study of numerical semigroups, in part because it leads to fast computations, though in the literature usually $n$ is chosen to be a nonzero element of $S$.

Next we present an useful relationship between Ap\'ery sets and divisors.

\begin{proposition}\label{bijection}
  The mapping $f:\Ap(S,n) \to \D(\m+n)\setminus \D(\m)$, $f(s)=\m+n-s$ is a bijection.
\end{proposition}
\begin{proof}
  Let us see that this map is well defined. Let $s\in \Ap(S,n)$. Then $\m+n-s\in\D(\m+n)$. On the other hand, as $\m+n-s=\m-(s-n)$, $\m+n-s\in S$ implies that $\m+n-s\not\in\D(\m)$. 
  
  The fact that $f$ is one to one is clear.

  Let now $\m+n-s\in S$ be a divisor of $\m+n$ (which implies that $s\in S$) that is not a divisor of $\m$. As $\m+n-s=\m-(s-n)$, the fact that $\m+n-s$ belongs to $S$ and is not a divisor of $\m$ implies that $s-n\not\in S$. It follows that we can take $s$ as a pre-image of $\m+n-s$, concluding in this way that $f$ is surjective.
\end{proof}

For symmetric numerical semigroups we can get an alternative description.
\begin{remark}
  Let $S$ be a symmetric numerical semigroup. Then $\D(\m +n)\setminus \D(\m)= \Ap(S,n)+\m -\ff(S)$.
\end{remark}
\begin{proof}
  Let $t\in \D(\m +n)\setminus \D(\m)$. Then $\m +n-t\in S$ and $\m -t\not\in S$. As $S$ is symmetric $\ff (S)-(\m +n-t)\not\in S$ and $\ff (S)-(\m -t)\in S$. Set $s=\ff (S)-(\m -t)\in S$. Then $s\in\Ap(S,n)$, and $t=s+\m -\ff (S)\in \Ap(S,n)+\m -\ff(S)$.

  For the other inclusion, take $s\in \Ap(S,n)$. Then $\m +n -(s+\m -\ff (S))=\ff (S)-(s-n)\in S$ ($s-n\not \in S$ and $S$ is symmetric), and $\m -(s+\m -\ff (S))=\ff (S)-s\not\in S$ (since $s\in S$ and $S$ is symmetric). Hence $\m +s-\ff (S)\in \D(\m +n)\setminus \D(\m)$.
\end{proof}

\section{Feng-Rao numbers of embedding dimension two numerical semigroups}\label{sec:basic-tools}

Let $S=\langle a, b\rangle$, with $a<b$ coprime integers greater than two. Let $c$ be the conductor of $S$. A well known result of Sylvester states that $c=ab-a-b+1$. Let $\m$ be an integer greater than or equal to $2c-1$.

Throughout this section, the letters $a$, $b$ and $\m$ shall be used with the above meanings.

This section is composed of various subsections. The first one, recalls some known facts for Weierstrass semigroups with two generators. Then we introduce some technical results that will be used in the rest of the paper. It is worth to highlight that among these, we present an explicit description of the Ap\'ery sets with restpect to any positive integer. Later we introduce a way to draw sets of integers that may help to follow the text remaining. The pictures, which show results produced with the package~\cite{numericalsgps}, have been created by using the \GAP~\cite{GAP4} package \texttt{IntPic}~\cite{intpic}. These type of images helped the authors to prove the results presented in this paper.  Next we show how to organize sets of divisors in triangles, and finally we discuss how to arrange them to obtain optimal configurations.

\subsection{Weierstrass semigroups with two generators}
For a base field ${\mathbb F}$ of characteristic zero, it is classically known that every numerical semigroup $S$ generated by two elements is actually a Weierstrass semigroup, in the sense that there exists an irreducible smooth projective algebraic curve $\chi$ and a point $P\in\chi$ such that the Weierstrass semigroup of $\chi$ at $P$ is precisely $S$ (see~\cite{Komeda}).

Unfortunately, this result is not proven to be true also in positive characteristic.  Nevertheless, there are sufficiently many examples of embedding dimension two numerical semigroups that are actually Weierstrass. In fact, provided ${\mathbb F}$ is a perfect field of positive characteristic (a finite field, in particular), one has that the plane curve given by the equation
\[
\alpha x^{a}+\beta y^{b}=\gamma
\]
has genus
\[
g=\frac{1}{2}(a-1)(b-1)
\]
where $\alpha,\beta,\gamma\in{\mathbb F}\setminus\{0\}$, $\gcd(a,b)=1$ and $\mathrm{char}\,{\mathbb F} \nmid a\cdot b$ (see~\cite{Stichtenoth}).

It is easy to check that the rational functions $x$ and $y$ have a unique pole at $P$, $P$ being the only point at infinity, of order $b$ and $a$ respectively, so that the semigroup $S=\langle a,b\rangle$ is contained in the Weierstrass semigroup $\Gamma$ of $\chi$ at $P$.  But since both semigroups $S$ and $\Gamma$ have the same genus, one concludes that $S=\Gamma$.

The above example shows that, for a given characteristic $p$, infinitely many embedding dimension two numerical semigroups are Weierstrass (those whose generators are none of them multiple of $p$). Conversely, a given semigroup $S=\langle a,b\rangle$ is Weierstrass for every characteristic $p$ but for a finite number of primes $p$ (namely, those prime factors of $a$ or $b$).

The above example does not work when the characteristic $p$ divides one of the generators. For example, if ${\mathbb F}={\mathbb F}_{2}$ and one considers the curve $x^{4}+y^{5}=1$, the genus turns out to be 0 (use for example the library {\tt brnoeth.lib}~\cite{brnoeth} of the computer algebra system {\sc Singular}~\cite{Singular}), so that the semigroup $S=\langle 4,5\rangle$ is not the Weierstrass semigroup of this curve at any of its points.

\subsection{Technical results}\label{subsec:technical}
We start by giving a procedure to decide when an integer belongs to $\langle a,b\rangle$. This criterion will be used many times in the rest of the paper. This is indeed a direct consequence of~\cite[Lemma 2.6]{NS}, and its proof is included for the sake of completeness.

\begin{lemma}\label{pertenencia-apery}
  Let $u\in \{0,\ldots,b-1\}$, and let $v$ be an integer.
  \begin{enumerate}[1.] 
   \item The integer $u a+vb\in S$ if and only if $v\ge 0$ (analogously, for $v\in \{0,\ldots,a-1\}$ and $u\in \mathbb Z$, $ua+vb\in S$ if and only if $u\ge 0$).
   \item If $ua+vb=u'a+v'b$ with $0\leq u'\leq b-1$ and $v'\in \mathbb Z$, then $u'=u$ and $v'=v$ (the same for $v'\in \{0,\ldots,a-1\}$, $u'\in \mathbb Z$) 
  \end{enumerate}
\end{lemma}
\begin{proof}
  \begin{enumerate}[1.]
  \item Clearly, if $v\ge 0$, then $ua+vb \in S$. For the converse, if $ua+vb\in S$, then there exist $u',v'\in \mathbb N$ such that $ua+vb=u'a+v'b$. Assume to the contrary that $v<0$. Then $ua=u'a+(v'-v)b$. As $b>a$, we get that $ua=u'a+(v'-v)b>(u'+v'-v)a$, and consequently $u>u'$. Hence $(u-u')a=(v'-v)b$. However, $\gcd(a,b)=1$, which implies that $b| u-u'$. Since $u-u'\in \{0,\ldots,b-1\}$, this forces $u=u'$, and then $v'-v=0$, contradicting $v<0, v'\in \mathbb N$.
  \item If $ua+vb=u'a+v'b$, as $u'a+(v'-v)b=ua$, we have too that $a$ divides $v'-v$ and we can write $v'-v=xa$, with $x\in \mathbb Z$, to obtain $(u'+xb)a=ua$. Since $0\leq u=u'+xb \leq b-1$ and $0\leq u'\leq b-1$, we deduce that $x=0$ and $u=u'$.
  \end{enumerate}
\end{proof}
We observe that for an integer $n$ given, there exist integers $u$ and $v$ such that $n=ua+vb$, since $a$ and $b$ are coprime. Furthermore, $u$ can be taken such that $0\le u<b$ (in fact, $u=na^{-1}\bmod b$) and, in this case, $u$ and $v$ are unique, by the above lemma.

Let $n$ be a positive integer. Next we give a description of $\Ap(S,n)$.
\begin{theorem}\label{desc-apery}
  Let $a$ and $b$ be coprime positive integers, and let $S=\langle a,b\rangle$. Let $n$ be an integer, and let $u$ and $v$ be integers with $0\le u<b$, such that $n=ua+vb$.  Then
  \[
  \Ap(S,n)=\begin{cases}
    \{\alpha a +\beta b\mid 0\le \alpha < u \hbox{ and } 0\le \beta < a+v\}, \hbox{ if } -a\le v < 0,\\
    \{\alpha a +\beta b\mid u\le \alpha < b \hbox{ and } 0\le \beta < v\} \\
    \qquad \cup \{\alpha a +\beta b\mid 0\le \alpha < u \hbox{ and } 0\le
    \beta < a+v\}, \hbox{ if } n\in S.
  \end{cases}
  \]
  In particular,
 $$\cardinal \Ap(S,n)=
\begin{cases}
  0 \mbox{ if } v<-a,\\
  u(a+v) \mbox{ if } -a \le v< 0,\\
  n \mbox{ otherwise}.
\end{cases}
$$
\end{theorem}
\begin{proof}
  Take $s=\alpha a + \beta b\in S$, with $0\le \alpha<b$, $\beta\ge 0$. Since $0\le u<b$ and $0\le \alpha < b$ implies that $-b<\alpha-u<b$, we have that $ s-n = (\alpha-u)a+(\beta-u)b\not\in S$ if and only if either $0\le\alpha-u<b$ and $\beta-v <0$, or $-b<\alpha-u<0$ and $\beta-v <a$. It follows that
  \begin{align*}
    \Ap(S,n)&=\{\alpha a +\beta b\mid (0\le \alpha -u < b \hbox{ and } \beta -v< 0) \hbox{ and } (0\le \alpha<b \hbox{ and } \beta\ge 0)\}\\
    &\cup \{\alpha a +\beta b\mid (-b< \alpha -u < 0 \hbox{ and } \beta -v<
    a) \hbox{ and } (0\le \alpha<b \hbox{ and } \beta\ge 0)\}.
  \end{align*}
  And the proof follows easily by studying the possible cases.
\end{proof}

Observe that we recover the well known fact that the Apéry set of an element $n$ in $S$ has cardinality $n$.

In light of Proposition~\ref{bijection}, the description of the Apéry sets for semigroups of embedding dimension two given in Theorem~\ref{desc-apery} yields a description of the set $\D(\m+n)\setminus \D(\m)$, that is, of the new divisors that $\m+n$ adds to those of $\m$.

To better understand the result below, we refer the reader to the figures in Example~\ref{ex:cor-div-alt}. 

\begin{corollary}\label{cor-div-alt}
  Let $n\in \mathbb N$, $n=ua+vb$ with $u\in\{0,\ldots, b-1\}$ and $v\in\mathbb Z$. Then
  \[
  \D(\m+n)\setminus\D(\m)=\begin{cases}
    \{\m+xa+yb \mid 0< x \le u \hbox{ and } -a< y \le v<0\}, \hbox{ if } v < 0,\\
    \{\m+xa+yb \mid u< x \le b \hbox{ and } -a< y \le v-a\}\\ \qquad \cup
    \{\m+xa+yb \mid 0<x\le u \hbox{ and } -a<y \le v\}, \hbox{ if } n\in S.
  \end{cases}
  \]
\end{corollary}
\begin{proof}
  Let $s\in \D(\m+n)\setminus\D(\m)$. By applying Proposition~\ref{bijection}, there exists $\alpha a+\beta b\in \Ap(S,n)$ such that
  \begin{align*}
    s=f(\alpha a +\beta b)= \m+ua+vb-(\alpha a +\beta b)
    &=\m+(u-\alpha)a+(v-\beta)b\\
    &=\m+(u-\alpha+b)a+(v-\beta-a)b.
  \end{align*}

  The inequalities in Theorem~\ref{desc-apery} involved in the description of $\Ap(S,n)$ can be rewritten as follows.
  \begin{itemize}
  \item $0\le \alpha < u$ is equivalent to $0<u-\alpha\le u$.
  \item $0\le \beta < a+v$ if and only if $-a< v -\beta \le v$.
  \item $u\le \alpha < b$ is the same as $u<u-\alpha+b\le b$.
  \item $0\le \beta < v$ is equivalent to $-a < v-\beta-a \le v-a$. 
  \item $0\le \beta < a+v$ corresponds to $-a<v-\beta\le v$.
  \end{itemize}

  From the above inequalities and Theorem~\ref{desc-apery} we obtain
  \[
  \D(\m+n)\setminus\D(\m)\subseteq\begin{cases}
    \{\m+xa+yb \mid 0< x \le u \hbox{ and } -a< y \le v<0\}, \hbox{ if } v < 0,\\
    \{\m+xa+yb \mid u< x \le b \hbox{ and } -a< y \le v-a\}\\ \qquad \cup
    \{\m+xa+yb \mid 0<x\le u \hbox{ and } -a<y \le v\}, \hbox{ if } n\in S.
  \end{cases}
  \]
  Now by Lemma~\ref{pertenencia-apery}, the set on the right hand side has the same cardinality as that of $\Ap(S,n)$, which by Proposition~\ref{bijection}, is the same as that of $\D(\m+n)\setminus\D(\m)$. Hence the equality holds.
\end{proof}

As a consequence of Corollary~\ref{cor-div-alt}, we also obtain
\begin{equation}\label{divisores-alt}
  \D(\m,\m+n) = \D(\m)\cup\begin{cases}
    \{\m+xa+yb \mid 0< x \le u \hbox{ and } -a< y \le v<0\}, \hbox{ if } v < 0,\\
    \{\m+xa+yb \mid u< x \le b \hbox{ and } -a< y \le v-a\}\\ \qquad 
    \cup \{\m+xa+yb \mid 0<x\le u \hbox{ and } -a<y \le v\}, \hbox{ if } n\in S,
  \end{cases}
\end{equation}
and this union is disjoint.

\subsection{A way to visualize integers}\label{sec:drawings}
Our purpose in this subsection is to construct a table where each integer appears exactly once and such that the way the integers are disposed helps the understanding of the problem treated in this paper, as well as many of the statements and proofs. Instead of the traditional arrangement of the integers in a straight line, we represent them in a an bi-infinite strip whose width depends on a given integer and the way the elements are presented depends on another integer which is smaller and coprime to the former one. The pictures, which show results produced with the package~\cite{numericalsgps}, have been produced by using the \GAP~\cite{GAP4} package \texttt{IntPic}~\cite{intpic}.

Let $a,b\in \mathbb{N}$ be coprime integers such that $a<b$. We shall construct a bi-infinite table such that each row has length $b$. For this purpose, we choose an integer $\oo$ which will work as the origin of a referential.  Take the row $\{\oo,\oo+a,\ldots, \oo+(b-1)a\}$ (which will work as the $x$-axis.)  The other rows are obtained by adding or subtracting multiples of $b$ in such a way that the $y$-axis is $\{\ldots,\oo-2b,\oo-b,\oo,\oo+b,\oo+2b,\ldots\}$. Similarly, each of the other columns consist of $\oo$ plus multiples of $b$ plus a certain fixed multiple of $a$ between $0$ and $(b-1)a$. It follows from Lemma~\ref{pertenencia-apery} that each integer appears exactly once in the table and, if we write $n=ua+vb$, with $u\in \{1,\ldots,b-1\}$ and $v\in \mathbb Z$, $u$ and $v$ may be seen as the $x$-coordinate and the $y$-coordinate, respectively.

Throughout the paper, the integers $a$ and $b$ will be taken as the minimal generators of an embedding dimension two numerical semigroup and $\oo$ is taken as $\m$.

The examples in this subsection show the relevant parts of pictures that have been produced by taking $a=11, b=29$ and $\oo=\m= 559$. In the next example the numbers are shown, so that the reader can easily verify which are the integers involved in the other examples. Colors (or gray tones) are used to highlight elements. For those elements belonging to more than one set whose elements are to be highlighted, we use gradients ranging through all the colors involved. The set $\{\m, \ldots, \m+b-1\}$ will be called the \emph{ground} with respect to $\m$, or simply ground when no possible confusion may arise.
\begin{example}\label{ex:axis_ground}
  Let $a=11, b=29$, $\oo=\m=559$.  The highlighted elements are those of the $x$ and $y$ axis and the ground. In this example, the steps of the ground have lengths $2$ or $3$.

\nopagebreak 
\centerline{\includegraphics[scale=0.8]{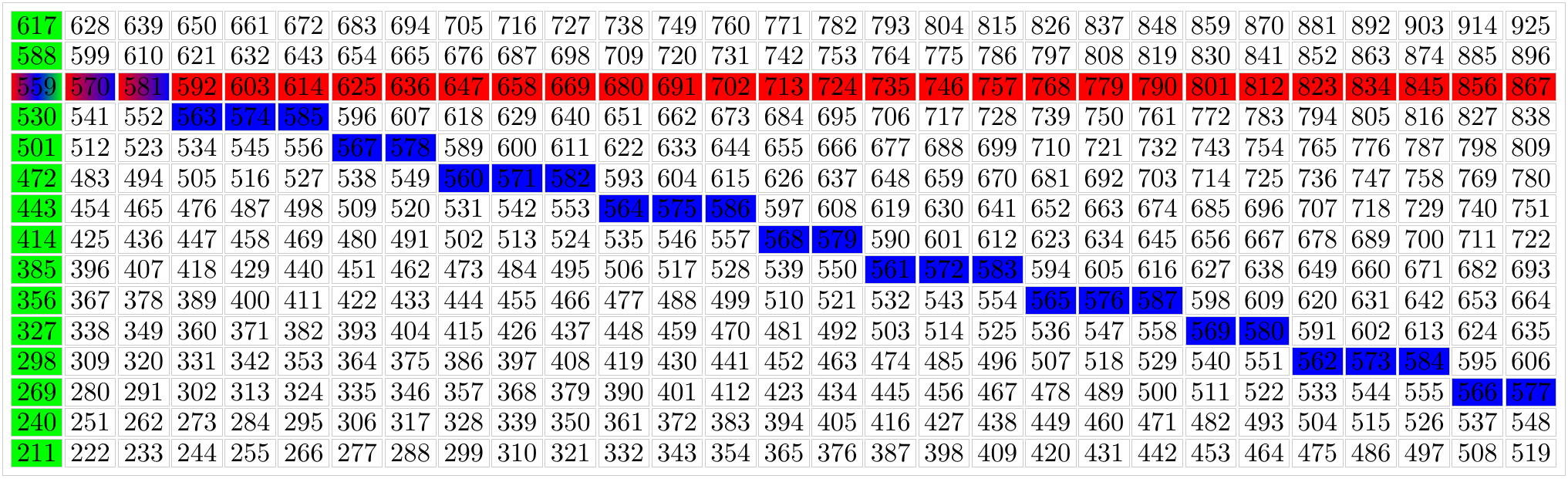}}
\nopagebreak
\end{example}

The following example gives us the (correct) impression that the way the elements are sorted in the construction of the table leads to disposing the divisors in a way that makes easy  to visualize and count them. The pictures are meant to illustrate the sets in Corollary~\ref{cor-div-alt}.

\begin{example}\label{ex:cor-div-alt}
The highlighted cells in the leftmost picture, corresponding to the case $n\not\in S$, are the divisors of $\m (=559)$, and the divisors of $\m+n (=\m+22a-8b=569)$ that are not divisors of $\m$. The highlighted cells in the picture on the right, correspond to the case $n\in S$, and are on the one hand the divisors of $\m (=559)$, and on the other hand, the divisors of $\m+n (=\m+2a+b=610)$ that are not divisors of $\m$; these consist of the union of two sets that are drawn by using different colors. 

\begin{center}
\includegraphics[scale=1.3]{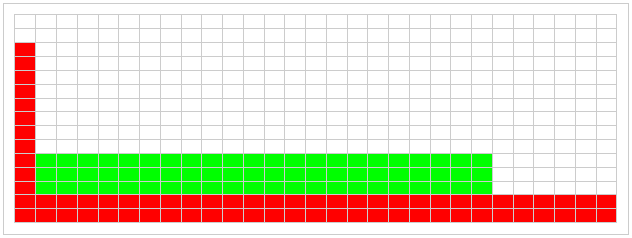}\quad
\includegraphics[scale=1.3]{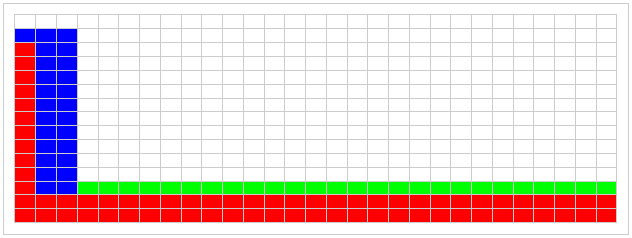}
\end{center}
\end{example}

\subsection{Ground, triangles and divisibility}

Every $x\in\{0,\ldots, b-1\}$ can be expressed as $ia\bmod b$ for a unique $i\in \{0,\ldots, b-1\}$, because $\gcd(a,b)=1$.

For $i\in \mathbb N$, we will write $\m \oplus i$ for $\m +(ia\bmod b)$, which is a rather convenient way to express uniquely the elements of the ground. 

In order to avoid the unnecessary parentheses, we will assume that the precedence of $\oplus$ is higher than the rest of binary operations. Thus, for instance, we will write $\m \oplus i+ha$ to refer to $(\m\oplus i)+ha$.

The divisors of an element in the ground, excluding the divisors of $\m$ are described in the following consequence of Corollary~\ref{cor-div-alt}.
\begin{corollary}\label{divisores-elemento-suelo}
  Let $i\in \{0,\ldots,b-1\}$. Then
  \[\D(\m\oplus i)\setminus\D(\m)=\{\m+xa+yb\mid 0<x\le i, -a<y
  \le-\peb[{ia}/{b}]\}.\]
\end{corollary}
\begin{proof}
  Just use Corollary~\ref{cor-div-alt} with $ia\bmod b=ia-\peb[ia/b]b$ ($u=i$, $v=-\peb[ia/b]$).
\end{proof}

We are going to see that if an element divides two elements in the ground and does not divide $\m$, then it divides all the elements between these two elements. First we give an example that may help to follow the proof.
\begin{example}\label{ex:tres-del-suelo}
The divisors of $\m\oplus 9$,  $\m\oplus 15$ and $\m\oplus 21$ are represented in the following picture.\\
  \centerline{
\includegraphics[scale=1.5]{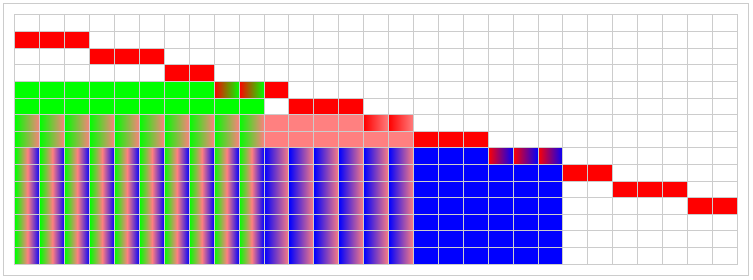}}
\end{example}

\begin{corollary}\label{tres-del-suelo}
  Let $i,j,k\in \{0,\ldots,b-1\}$ with $i<j<k$. Then
  \[ (\D(\m\oplus i)\cap \D(\m\oplus k))\setminus \D(\m)\subseteq \D(\m\oplus
  j)\setminus \D(\m).\]
\end{corollary}
\begin{proof}
  Following Corollary~\ref{divisores-elemento-suelo}, if we take $s\in (\D(\m\oplus i)\cap \D(\m\oplus k))\setminus \D(\m)$, then $s=\m+xa+yb$ with $0<x \le i$ and $-a <y \le -\peb[ka/b]$. Thus $0<x \le i<j$ and $-a <y \le -\peb[ka/b]\le -\peb[ja/b]$, and so by using again Corollary~\ref{divisores-elemento-suelo}, we obtain that $x\in \D(\m\oplus j)\setminus \D(\m)$.
\end{proof}

Given $n,n'\in \mathbb N$, $n'\le_S n$ if and only if $\D(\m+n')\subseteq \D(\m+n)$, or equivalently, $\D(\m+n')\cap [\m,\infty)\subseteq \D(\m+n)\cap [\m,\infty)$. That is $n-n'\in S$ ($n'$ divides $n$ with respect to $S$) if and only if $\D(\m+n')\cap [\m,\infty)$, is included in $\D(\m+n)\cap [\m,\infty)$. We say that $\D(\m+n)\cap [\m,\infty)$ is the \emph{triangle associated} to $n$, and $\D(\m+n)\cap [\m,\m+b)$ is its \emph{base}. Also we will refer to $n$ as the \emph{upper vertex} of the triangle. Thus we have shown that $n\le_S n'$ if and only if the triangle associated to $n$ is included in that associated to $n'$. We are going to see that we do not need to compare the whole triangles, but just the bases.
\begin{example}\label{ex:triangles}
The following pictures illustrate the two existing kinds of triangles (depending on having $\m$ in the base or not). The bases and upper vertices are highlighted.\\
\centerline{\includegraphics[scale=1.3]{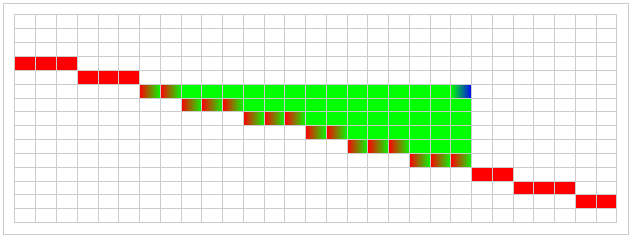}\quad
\includegraphics[scale=1.3]{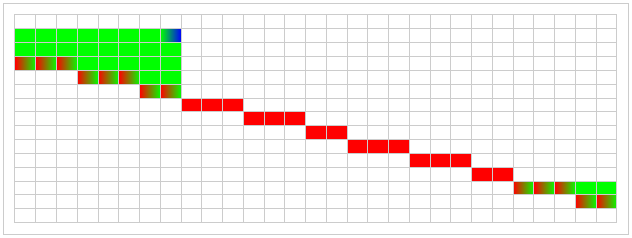}}
\end{example}

An \emph{interval} of the ground is a subset $L$ of $[\m,\m+b)$ of the form $L=\{\m\oplus i,\dots, \m\oplus (i+h)\}$, with $i,h\in\{0,\ldots,b-1\}$. Observe that if $ia\bmod b\ge a$, then $\m+ia\bmod b-a=\m\oplus(i-1)\in (\D(L)\cap[\m,\m+b))\setminus L$ and $L$ is non amenable. An interval of the ground that happens to be an amenable set is said to be an \emph{amenable interval}.

\begin{example}\label{ex:amenable_intervals}
The two existing types of amenable intervals are illustrated in the figures below. They have been obtained using, respectively, $i=6, h=8$ and $i=22, h=11$.\\
\centerline{\includegraphics[scale=1.3]{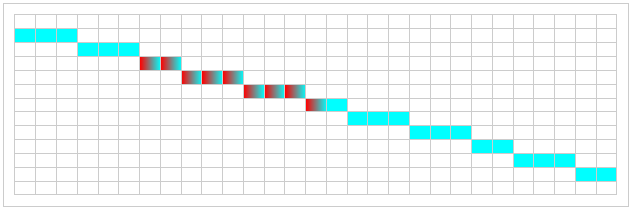}\quad
\includegraphics[scale=1.3]{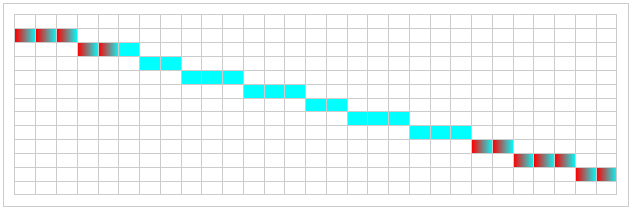}}
\end{example}
As we see next, every amenable interval is realizable as the base of a triangle.
\begin{lemma}\label{desc-suelo-D}
  Let $L=\{\m\oplus i,\dots, \m\oplus (i+h)\}$ be an amenable interval, with $0\le i,h<b$. Then $\D(\m\oplus i +ha)\cap [\m,\m+b)=L$.
\end{lemma}
\begin{proof}
The divisors of $\m\oplus i +ha$ can be expressed as $\m\oplus i+ha-xa-yb$ with $x,y\in \mathbb N$.

Let $x,y\in \mathbb N$. Next we use the division algorithm to manipulate $\m+ia\bmod b+ha-xa-yb$.
\begin{align*}
\m+ia\bmod b+ha-xa-yb & =\m+ ia\bmod b +\peb[{(h-x)a}/{b}]b+(h-x)a\bmod b-yb\\ & =\m+ia\bmod b+(h-x)a\bmod b +(\peb[{(h-x)a}/{b}]-y)b.
\end{align*}
Hence
\begin{equation}\label{align:desc-suelo-D}
\m+ia\bmod b+ha-xa-yb= \m+ (i+h-x)a\bmod b + Yb,
\end{equation}
where 
\[
Y = \left\{\begin{array}{lcl}
\peb[{(h-x)a}/{b}]-y &\mbox{ if }&0\le ia\bmod b+(h-x)a\bmod b<b,\\
\peb[{(h-x)a}/{b}]-y+1 &\mbox{ if }&b\le ia\bmod b+(h-x)a\bmod b<2b.
\end{array}\right.
\]
Hence
\[ \m+ia\bmod b+ha-(\m+ (i+h-x)a\bmod b )=xa+(y+Y)b.\]

As the elements of $L$ are in the ground and may be written as $\m\oplus (i+h-x)$, with $x\in\{0,\ldots ,h\}$, the above equation proves that $L\subseteq \D(\m\oplus i +ha)\cap [\m,\m+b)$.
\medskip

In order to prove the reverse inclusion recall that as $L$ is amenable, we have that $ia\bmod b<a$. Thus
\begin{multline*}
\m\oplus i+ha-xa-yb=\m+(ia\bmod
b)+ha-xa-yb\\<\m+a+ha-xa-yb=\m+(1+h-x)a-yb.
\end{multline*}
If $x>h$, then we get a divisor that is smaller than $\m$, therefore we may assume that $x\in\{0,\ldots ,h\}$. Now it suffices to use again Equation~\eqref{align:desc-suelo-D}, to see that $\m\oplus i+ha-xa-yb\in [m,m+b)$ if and only if $Y=0$, and then $ \m\oplus i+ha-xa-yb=  \m+ (i+h-x)a\bmod b \in \{m\oplus i,\ldots, m\oplus (i+h)\}=L$.
\end{proof}
For $h>b-1$, $\D(\m\oplus i+(b-1)a)\subseteq \D(\m\oplus i+ha)$. By Lemma~\ref{desc-suelo-D}, $\D(\m\oplus i+(b-1)a)\cap[\m,\m+b)=\{\m,\ldots,\m+b-1\}$, and consequently $\D(\m+n)\cap [\m,\m+b)=\{\m,\ldots, \m+b-1\}$.



\begin{remark}\label{hat-en-S}
Let $n\in \mathbb N$. Then $n=ia\bmod b+ha$, with $i=(n\bmod a)a^{-1}\bmod b$ and $h=\peb[\frac{n}{a}]$. This is because $n=\peb[{n}/{a}]a+n\bmod a=ha+ia\bmod b$. Observe that 
\[ia\bmod b=n-ha<a,\] 
and thus $L=\{\m\oplus i,\ldots, \m\oplus(i+h)\}$ is an amenable interval. Moreover, by Lemma \ref{desc-suelo-D}, $L=\D(\m+n)\cap [\m,\m+b)$, and if $h<b$, $L\neq \{\m,\ldots, \m+b-1\}$. 


\begin{itemize}
\item If $\m\in L$, by Lemma~\ref{desc-suelo-D}, $\m\in \D(\m\oplus i +ha)=\D(\m+n)$. Thus $\m+n-\m=n\in S$. The converse is trivially true. Hence $\m\in L$ if and only if $n\in S$. In this setting, if $i\neq 0$, $L$ can be written as \[L=\{\m,\ldots, \m\oplus(i+h-b)\}\cup \{\m\oplus i,\ldots, \m\oplus(b-1)\}.\] If $i=0$, then $L=\{\m, \ldots, \m \oplus h\}$.
\item If $\m\notin L$ ($n\not \in S$), then $h+i<b$ (since if $i+h\geq b$, $\m=\m\oplus b\in L$), and $0<ia\bmod b$ (since otherwise, $n=ha\in S$).
\end{itemize}
\end{remark}

%
%
%

Observe that if $M$ is an amenable set and $L=M\cap[\m,\m+b)=\{\m,\ldots, \m+b-1\}$, then according to Lemma~\ref{sombra}, $\cardinal \D(M)=\cardinal (\D(\m,\ldots,\m+b-1)\cap[0,\m))+\cardinal M$. Hence whenever we add an element to $M$ so that it remains amenable, the resulting number of divisors is increased just by one. Thus we are mainly interested in the case $M\cap [\m,\m+b)\subsetneq \{\m,\ldots,\m+b-1\}$.

\begin{lemma}\label{divisibilidad-y-suelos}
  Let $n,n'\in \mathbb N$ with $\D(\m+n)\cap [\m,\m+b)\neq \{\m,\ldots,\m+b-1\}$. Then $n'\le_S n$ if and only if $\D(\m+n')\cap[\m,\m+b)\subseteq \D(\m+n)\cap [\m,\m+b)$.
\end{lemma}
\begin{proof}
  If $n'\le_S n$, then as it was already mentioned above, trivially $\D(\m+n')\subseteq \D(\m+n)$.

  For the converse, let $i,h,i',h'$ be as in Remark~\ref{hat-en-S}, such that $n=ia\bmod b+ha$ and $n'=i'a\bmod b+h'a$. Notice that $n-n'=(h+i-h'-i')a+\left( \peb[{i'a}/b]-\peb[{ia}/b]\right)b =(h+i-h'-i'-b)a+\left( \peb[{i'a}/b]-\peb[{ia}/b]+a\right)b$.

In view of Remark~\ref{hat-en-S}, $\D(\m+n')\cap[\m,\m+b) = \{\m\oplus i',\dots, \m\oplus(i'+h')\}$, and $\D(\m+n)\cap [\m,\m+b) = \{\m\oplus i,\ldots, \m\oplus(i+h) \}$.

  As $\D(\m+n')\cap[\m,\m+b)\subseteq \D(\m+n)\cap [\m,\m+b)\neq \{\m,\ldots,\m+b-1\}$, we deduce $\{\m\oplus i',\dots, \m\oplus(i'+h')\}\subseteq \{\m\oplus i,\ldots, \m\oplus(i+h) \}$ and $h<b-1$. The following cases may occur.
  \begin{enumerate}[1.]
  \item If $i+h<b$, then $i\le i'$ and $h'+i'\le h+i$. Hence $n-n'\in S$, because $h+i-h'-i'\ge 0$ and $\peb[{i'a}/b]-\peb[{ia}/b]\ge 0$.

  \item If $i+h\ge b$, then $\{\m\oplus i,\dots, \m\oplus (i+h)\} =\{\m,\ldots, \m\oplus(i+h-b)\}\cup \{\m\oplus i,\ldots, \m\oplus(b-1)\}$, and we have to distinguish three sub-cases.
    \begin{enumerate}[i.]
    \item $0\le i'\le i'+h'\le i+h-b$. In this setting, $n-n'=(h+i-b-h'-i')a+\left(\peb[{i'a}/b]-\peb[{ia}/b]+a\right)b$, which is in $S$, since $h+i-b-h'-i'\ge 0$ and $\peb[{i'a}/b]-\peb[{ia}/b]+a\ge 0$.

    \item $i\le i'\le i'+h'<b(\le i+h)$. Now, $n-n'=(h+i-h'-i')a+\left(\peb[{i'a}/b]-\peb[{ia}/b]\right)b$, which is in $S$.

    \item $i\le i'$ and $0\le i'+h'-b\le i+h-b<i-1$. In this case, $n-n'=(h+i-h'-i')a+\left(\peb[{i'a}/b]-\peb[{ia}/b]\right)b$, which belongs to $S$.\qedhere
    \end{enumerate}
  \end{enumerate}
\end{proof}

Next result tells us that triangles are in some sense maximal amenable sets with respect to their base.

\begin{corollary}\label{cabeza-es-triangulo}
  Let $L=\{ \m\oplus i,\ldots, \m\oplus(i+h)\}\neq \{\m,\ldots,\m+b-1\}$, $0\le i,h <b$, be an amenable interval. Let $M$ be an amenable set such that $M\cap[0,\m+b)\subseteq L$. Then $M\subseteq \D(\m\oplus i +ha)$.
\end{corollary}
\begin{proof}
  Let $n=ia\bmod b+ha$ and take $\m+n'\in M$. As $M$ is amenable, $\D(\m+n')\subseteq M$, and $\D(\m+n')\cap [\m,\m+b)\subseteq L$. In light of Lemmas~\ref{desc-suelo-D} and ~\ref{divisibilidad-y-suelos}, we have $n'\le_S n$. Hence $\m+n'\in \D(\m+n)\cap[\m,\infty) \subseteq \D(\m\oplus i +ha)$.
\end{proof}

\subsection{Moving triangles and optimal configurations}

The results obtained so far allow us to assert that an amenable set is a union of triangles (not necessarily disjoint). We see in this section how to organize these triangles so that we get a configuration with the least possible number of divisors. First we prove that the size of the triangles increases as we increase their upper vertex.

\begin{lemma}\label{triangulos}
  Let $n, n'\in\mathbb N$, $n\le n'$. Then $\cardinal (\D(\m+n)\cap[\m,\infty))\leq \cardinal (\D(\m+n')\cap[\m,\infty))$. 
\end{lemma}
\begin{proof}
  If $\m+t\in \D(\m+n)\cap[\m,\infty)$, then $t=n-s$ with $s\in S$. Since $t+(n'-n)=n'-s$, we get $\m+t+(n'-n)\in \D(\m+n')\cap[\m,\infty)$.
\end{proof}

As we saw above, triangles are uniquely determined by their bases, which are amenable intervals. Moreover, the number of divisors of the elements in the triangles smaller than $\m$ depends only on the elements in their bases. We introduce a way to arrange amenable intervals that allows us to handle easily the elements in the ground of an amenable set.

Given $L$ and $L'$ amenable intervals, we write $L\prec L'$ if $L\cup L'$ is not an amenable interval and either
\begin{itemize}
\item $\m\in L$ or
\item $\m\not\in L\cup L'$, and for all $\m\oplus x\in L$ and every $\m\oplus y\in L'$, $x< y$ (the condition $L\cup L'$ is not an amenable interval then forces $x+1<y$, and also that $L\cap L'=\emptyset$).
\end{itemize}

\begin{remark}
  Let $M$ be an amenable set. The set $L=M\cap [\m,\m+b)$ can be expressed as union of disjoint amenable intervals, $L=L_1\cup\dots \cup L_t$, such that $L_1\prec L_2\prec \cdots \prec L_t$.
\end{remark}

According to Proposition~\ref{condiciones-ms}, when looking for an optimal configuration, we may choose $M$ amenable with $\m\in M$. This is why in the following results we may impose $\m\in L_1$ without loosing generality. In light of this, we will also assume that $\m\ \oplus(b-1)\notin L_t$ for $t>1$, since we will take $L_1\prec L_t$.

\begin{proposition}\label{solo-cuentan-los-contiguos}
  Let $L_1,\ldots, L_t$ be a sequence of amenable intervals with $\m\in L_1\prec \cdots \prec L_t$. For $1<i<t$, \[\D(L_i)\setminus\D(L_1\cup \cdots \cup L_{i-1}\cup L_{i+1}\cup \cdots\cup L_t)=\D(L_i)\setminus \D(\{\m\}\cup L_{i-1}\cup L_{i+1}).\]
\end{proposition}
\begin{proof}
  The inclusion $\D(L_i)\setminus\D(L_1\cup \cdots \cup L_{i-1}\cup L_{i+1}\cup \cdots\cup L_t)\subseteq \D(L_i)\setminus \D(\{\m\}\cup L_{i-1}\cup L_{i+1})$ is trivial. For the other inclusion, let $n_i\in \mathbb N$ be such that $\D(\m+n_i)\cap [\m,\m+b)=L_i$ (Lemma~\ref{desc-suelo-D}). As $i>1$, $\m\not\in \D(L_i)$, and thus $n_i\not\in S$ (Remark~\ref{hat-en-S}).

  Let $s\in \D(L_i)\setminus \D(\{\m\}\cup L_{i-1}\cup L_{i+1})$. Assume that $s\in \D(L_j)$ with $j\notin \{i-1,i,i+1\}$. Then there exist $u,v\in \{0,\ldots,b-1\}$, such that $\m\oplus u\in L_j$, $\m\oplus v\in L_i$ and $s\in(\D(\m\oplus u)\cap \D(\m\oplus v))\setminus \D(\m)$. If $u<v$, then from the hypothesis it easily follows that for all $w\in\{0,\ldots,b-1\}$ with $\m\oplus w\in L_{i-1}$, we have $u<w<v$. Then, by taking any of such $w$ and by Corollary~\ref{tres-del-suelo}, we deduce that $s\in \D(\m\oplus w)\setminus \D(\m)\subseteq \D(L_{i-1})$, a contradiction. If $u>v$, we proceed analogously but with $L_{i+1}$.
\end{proof}

This result allows us to focus in what happens when we have three disjoint triangles and we want to move the one in the middle. Our aim is to change $\D(\m,\m+n_1,\m+n_2,\m+n_3)$ with $\D(\m,\m+n_1+(h_2+1)a, \m+n_3)$. We are going to see that in this way, the number of divisors below $\m$ decreases, while we get more over $\m$ (see the picture in Example~\ref{ex:capacidad-cambio}).

First we see how many new divisors $\m +n_2$ adds to those of $\m +n_1$ and $\m +n_3$. To see this we will use \eqref{divisores-alt}.

\begin{lemma}\label{de-mas-enmedio}
  Let $n_1, n_2, n_3\in \mathbb N$, and set $L_j=\D(\m+n_j)\cap[\m,\m+b)$, $j\in \{1,2,3\}$. Assume that $\m\in L_1\prec L_2\prec L_3$.  Write $n_j=u_ja+v_jb$ with $u_j\in\{0,\ldots, b-1\}$ and $v_j\in \mathbb Z$, $j\in\{1,2,3\}$. Then $u_1<u_2$, $v_3<v_2<0$ and \[\D(\m+n_1,\m+n_2,\m+n_3)\setminus \D(\m+n_1,\m+n_3)=\{ \m+xa+yb\mid u_1<x\le u_2, v_3< y\le v_2\}.\]
\end{lemma}

\begin{proof}
  Observe that $\D(\m+n_1,\m+n_2,\m+n_3)\setminus \D(\m+n_1,\m+n_3)= \D(\m,\m+n_2)\setminus \D(\m+n_1,\m+n_3)= (\D(\m,\m+n_2)\setminus \D(\m,\m+n_1))\cap (\D(\m,\m+n_2)\setminus \D(\m,\m+n_3))$.

  Let $i_j,h_j\in \{0,\ldots, b-1\}$ be such that $n_j=i_ja\bmod b+ h_ja$ as in Remark~\ref{hat-en-S}. Then we have that $L_j=\{\m\oplus i_j,\ldots, \m\oplus (i_j+h_j)\}$. The condition $L_1\prec L_2\prec L_3$, implies that $n_2,n_3\not\in S$. And as $\m\in L_1$, by Remark~\ref{hat-en-S} again, if $i_1\neq 0$, $i_1+h_1\ge b$. Hence $u_2=i_2+h_2$, $u_3=i_3+h_3$, $v_2=-\peb[i_2a/b]$ and $v_3=-\peb[i_3a/b]$. If $i_1\neq0$, then $u_1=i_1+h_1-b$, $v_1=a-\peb[i_1a/b]$, and if $i_1=0$, $u_1=h_1$ and $v_1=0$. Hence $u_1<u_2<u_3$ and $v_1-a<v_3<v_2<0\le v_1$.

  The proof now follows by using \eqref{divisores-alt}.
\end{proof}
\begin{example}\label{ex:de-mas-enmedio}
The following picture illustrates the sets involved in the preceding lemma.  It was made taking $n_1 = 6a+b,n_2 = 15a-4b$ and $n_3 = 25a-7b$.\\ 
\centerline{\includegraphics[scale=1.5]{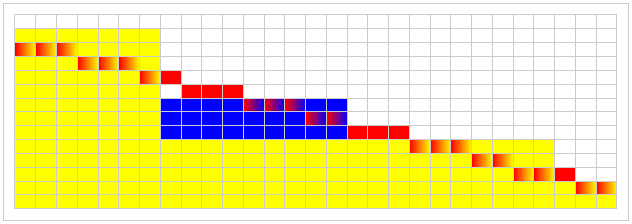}}
\end{example}

Now we see how many new divisors $\m +n_1+ka$ and $\m + ka$ add to those of $\m +n_1$ and $\m +n_3$.

\begin{lemma}\label{de-mas-pegado}
  Let $n_1, n_3 \in \mathbb N$, and set $L_j=\D(\m+n_j)\cap[\m,\m+b)$, $j\in \{1,3\}$. Assume that $\m\in L_1\prec L_3$, and let $k\in \mathbb N$ be such that $\D(\m+n_1+ka)\cap[\m,\m+b)\prec L_3$. Write $n_j=u_ja+v_jb$ with $u_j\in\{0,\ldots, b-1\}$ and $v_j\in \mathbb Z$, $j\in\{1,3\}$. Then \[\D(\m+n_1+ka,\m+n_3)\setminus \D(\m+n_1,\m+n_3)=\{ \m+xa+yb\mid u_1< x\le u_1+k, v_3 <y \le v_1\}.\]
\end{lemma}
\begin{proof}
  As in Lemma~\ref{de-mas-enmedio} the proof follows from \eqref{divisores-alt}.
\end{proof}
\begin{example}\label{ex:de-mas-pegado}
The following picture illustrates the sets involved in the preceeding lemma.  It was made by taking  $n_1 = 6a+b$, $k = 9$ and $n_3 = 25a-7b$.\\ 
\centerline{\includegraphics[scale=1.5]{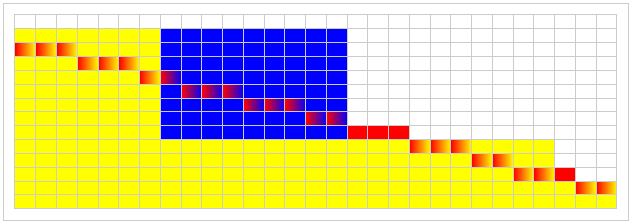}}
\end{example}

\begin{lemma}\label{de-mult-a}
  Let $n_1, n_3 \in \mathbb N$, and set $L_j=\D(\m+n_j)\cap[\m,\m+b)$, $j\in \{1,3\}$. Write $n_j=u_ja+v_jb$ with $u_j\in\{0,\ldots, b-1\}$ and $v_j\in \mathbb Z$, $j\in\{1,3\}$. Assume that $\m\in L_1\prec L_3$, and let $k\in \mathbb N$ be such that $\D(\m+ka)\cap[\m,\m+b)\prec L_3$ and $k\ge u_1$. Then
  \[\D(\m+ka,\m+n_3)\setminus \D(\m+n_1,\m+n_3)=\{ \m+xa+yb\mid u_1< x\le k, v_3 <y \le 0\}.\]
\end{lemma}
\begin{proof}
  Again, the proof follows from \eqref{divisores-alt}.
\end{proof}
\begin{example}\label{ex:de-mult-a}
The following figure represents the sets involved in Lemma \ref{de-mult-a}.  We used $n_1 = 6a+b$, $k = 11$ and $n_3 = 25a-7b$.\\ 
\centerline{\includegraphics[scale=1.5]{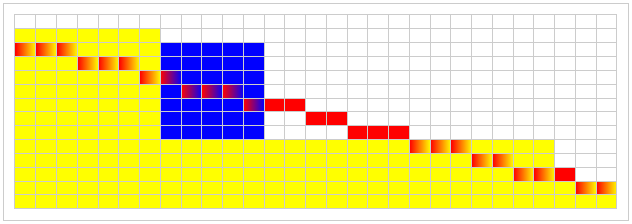}}
\end{example}

The next task is to see that the number of divisors below $\m$ decreases when we change $n_2$ with $n_1+(h_2+1)a$.
\begin{example}\label{ex:capacidad-cambio}
Let us compare these sets in an example. Set $n_1 = 146, n_2 = 75$ and $n_3 = 54$, which corresponds to $u_1 = n_1a^{-1} \bmod b = 8, v_1 = 2, u_2 =20, v_2 =-5, u_3 =26, v_3 =-8, i_1 =(n_1 \bmod a)a^{-1} \bmod b=24, h_1 =13, i_2 =14, h_2 =6, i_3 =22, h_3 =4$. \\
\centerline{\includegraphics[scale=1.5]{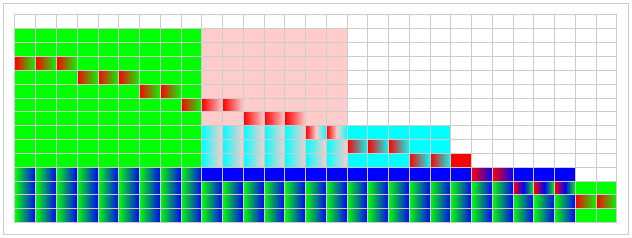}}
\end{example}

\begin{lemma}\label{capacidad-cambio}
  Let $n_1, n_2, n_3\in \mathbb N$, and set $L_j=\D(\m+n_j)\cap[\m,\m+b)$, $j\in \{1,2,3\}$. Assume that $\m\in L_1\prec L_2\prec L_3$. Let $i_j,h_j\in \{0,\ldots, b-1\}$ be such that $n_j=i_ja\bmod b+h_ja$ (as in Remark~\ref{hat-en-S}). Then \[\cardinal (\D(\m+n_1,\m+n_2,\m+n_3)\cap [\m,\infty))\le \cardinal (\D(\m+n_1+(h_2+1)a, \m+n_3)\cap [\m,\infty)).\]
\end{lemma}
\begin{proof}
  By using Lemma~\ref{divisibilidad-y-suelos} and that $L_1\prec L_2\prec L_3$, it is easy to prove that $\cardinal (\D(\m+n_1,\m+n_2,\m+n_3)\cap [\m,\infty))=\sum_{j=1}^3\cardinal (\D(\m+n_j)\cap[\m,\infty))$ and $\cardinal (\D(\m+n_1+(h_2+1)a, \m+n_3)\cap [\m,\infty))=\cardinal (\D(\m+n_1+(h_2+1)a)\cap[\m,\infty))+\cardinal (\D(\m+n_3)\cap[\m,\infty))$.

  As $\m\in L_1\prec L_2$, we get $\m\not\in L_2$, and thus by Remark \ref{hat-en-S}, $n_2\notin S$. Therefore, Remark \ref{hat-en-S}, asserts that $0<i_2a\bmod b<a$. Thus, $h_2a<n_2=i_2a\bmod b +h_2a<(h_2+1)a$. Hence, from Lemma~\ref{triangulos}, $\cardinal (\D(\m+n_2)\cap[\m,\infty))\le\cardinal (\D(\m+(h_2+1)a)\cap[\m,\infty))$.

  For every element $\m+x$ in $\D(\m+(h_2+1)a)\cap[\m,\infty)$, we have $\m+n_1+x\in \D(\m+n_1+(h_2+1)a)\setminus\D(\m+n_1)$. Hence $\cardinal (\D(\m+n_1+(h_1+1)a)\cap[\m,\infty))-\cardinal (\D(\m+n_1)\cap[\m,\infty))\ge \cardinal (\D(\m+(h_2+1)a)\cap[\m,\infty))\ge \cardinal (\D(\m+n_2)\cap[\m,\infty))$. This proves $\cardinal (\D(\m+n_1+(h_2+1)a)\cap[\m,\infty))\ge \cardinal (\D(\m+n_1)\cap[\m,\infty))+\cardinal (\D(\m+n_2)\cap[\m,\infty))$.
\end{proof}

And now we show that we gain divisors over $\m$.

\begin{lemma}\label{carga-cambio}
  Let $n_1, n_2, n_3\in \mathbb N$, and set $L_j=\D(\m+n_j)\cap[\m,\m+b)$, $j\in \{1,2,3\}$. Assume that $\m\in L_1\prec L_2\prec L_3$. Let $i_j,h_j\in \{0,\ldots, b-1\}$ be such $n_j=i_ja\bmod b+h_ja$ (as in Remark~\ref{hat-en-S}). Then
  \begin{multline*}
    \cardinal ((\D(\m+n_1,\m+n_2,\m+n_3)\setminus\D(\m+n_1,\m+n_3))\cap [0,\m))\\ 
      \ge \cardinal ((\D(\m+n_1+(h_2+1)a, \m+n_3)\setminus\D(\m+n_1,\m+n_3))\cap [0,\m)).
  \end{multline*}
\end{lemma}
\begin{proof}
  For $j\in\{1,2,3\}$, let $u_j\in \{0,\ldots,b-1\}$ and $v_j\in \mathbb Z$ such that $n_j=u_ja+v_jb$. Then, as above, either $u_1=i_1+h_1-b$ and $v_1=a-\peb[i_1a/b]$ ($i_1\neq 0$), or $u_1=h_1$ and $v_1=0$ ($i_1=0$). Also $u_2=i_2+h_2$, $u_3=i_3+h_3$, $v_2=-\peb[i_2a/b]$ and $v_3=-\peb[i_3a/b]$. Remark~\ref{hat-en-S} describes both $L_1$ and $L_2$, and as a consequence of $L_1\prec L_2$, we get $u_1<i_2$. Thus $u_1+h_2+1\le i_2+h_2=u_2$. Let
  \[\begin{array}{l}
    A=\D(\m+n_1,\m+n_2,\m+n_3)\setminus\D(\m+n_1,\m+n_3),\\
    B=\D(\m+(u_1+h_2+1)a, \m+n_3)\setminus\D(\m+n_1,\m+n_3),\\
    C=\D(\m+u_2a, \m+n_3)\setminus\D(\m+n_1,\m+n_3).\\
  \end{array}\]
  From Lemmas~\ref{de-mas-enmedio} and~\ref{de-mult-a}, we deduce that
  \[\begin{array}{l}
    A=\{ \m+xa+yb\mid u_1<x\le u_2, v_3<y\le v_2\},\\
    B=\{ \m+xa+yb\mid u_1< x\le u_1+h_2+1, v_3<y\le 0\},\\
    C=\{ \m+xa+yb\mid u_1<x\le u_2, v_3<y\le 0\}.\\
  \end{array}\]
  Notice that $A\subseteq C$, and $B\subseteq C$. Also
  \[\begin{array}{l}
    C\setminus A=\{ \m+xa+yb\mid u_1<x\le u_2, v_2<y\le 0\},\\
    C\setminus B=\{ \m+xa+yb\mid u_1+h_2+1<x\le u_2, v_3<y\le 0\}.\\
  \end{array}\]
  Define
  \[
  \begin{array}{l}
    R_A=\{\m+xa+yb \mid u_1<x \le u_2-h_2-1, v_2<y\le 0\},\\
    R_B=\{\m+xa+yb \mid u_1+h_2+1<x\le u_2, v_3<y \le v_3-v_2\}.
  \end{array}
  \]
  Then we can write
  \[R_A=v_A+D, \quad R_B=v_B+D,\] where $D=\{xa+yb\mid u_1-u_2+h_2+1<x\le 0,
  0\le y<-v_2\}$, $v_A=\m+(u_2-h_2-1)a+(v_2+1)b$ and $v_B=\m+u_2a+(v_3+1)b.$
\medskip

The following figure illustrates the regions involved in this proof. The one on the right corresponds to the region $C$ where $A$, $B$, $R_A$, $R_B$, $R_A\cap[0,\m)$ and $R_B\cap[0,\m)$ are highlighted.

\centerline{
\includegraphics[scale=0.7]{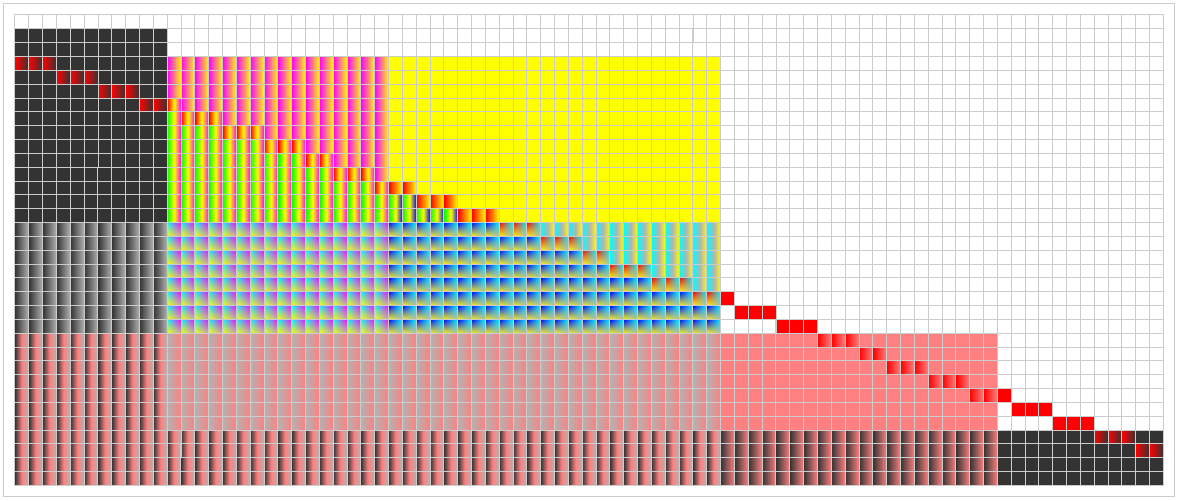}\quad
\includegraphics[scale=1.19]{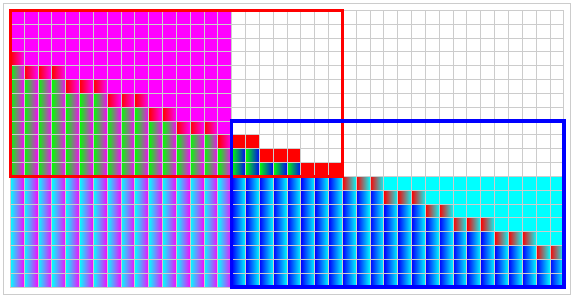}
}

  We claim that $(C\setminus A)\cap [0,\m)\subseteq R_A\cap [0,\m)$. Let us prove that $(C\setminus A)\setminus R_A\subseteq [\m,\infty)$. To this end, observe that $(C\setminus A)\setminus R_A= \{ \m+xa+yb\mid u_2-h_2-1<x \le u_2, v_2<y\le 0\} \subseteq v_A+\mathbb N$. In addition, $v_A= \m+u_2a-h_2a-a+v_2 b+b= \m-h_2a+n_2-a+b= \m+i_2a\bmod b-a+b$. As $L_1\prec L_2$, we have $\m\notin L_2$, whence by Remark \ref{hat-en-S}, $n_2\not \in S$, and by Remark \ref{hat-en-S}, $0<i_2a\bmod b<a$. Thus, $\m+b-a< v_A < \m+b$.

  Since $R_B\subseteq C\setminus B$, we get trivially that $R_B\cap [0,\m)\subseteq (C\setminus B)\cap [0,\m)$.

  Finally we prove that $R_A\cap[0,\m)\hookrightarrow R_B\cap [0,\m)$. First, $v_B=\m+(i_2+h_2)a-\peb[i_3a/b]b+b= \m+(i_2+h_2-i_3)a+i_3a\bmod b +b$. Again, by Remark \ref{hat-en-S}, $i_3a\bmod b<a$, and thus $v_B<\m+(i_2+h_2+1-i_3)a+b$. Moreover, $L_2\prec L_3$, whence $i_2+h_2+1<i_3$, and consequently $v_B<\m-a+b<v_A$. For every $n\in R_A\cap[0,\m)$, $n=v_A+x$, $x\in D$, and $v_A+x<\m$. Hence $v_B+x<v_A+x<\m$. This implies that the map $R_A\to R_B$, $v_A+x\mapsto v_B+x$ is injective and maps elements in $R_A\cap[0,\m)$ to elements in $R_B\cap[0,\m)$.

  Therefore, $\cardinal ((C\setminus A)\cap [0,\m))\le \cardinal (R_A\cap [0,\m))\le \cardinal (R_B\cap [0,\m))\le \cardinal ((C\setminus B)\cap [0,\m))$. Hence $\cardinal (B\cap [0,\m))\le \cardinal (A\cap [0,\m))$. In view of Lemma~\ref{de-mas-pegado}, $(\D(\m+n_1+(h_2+1)a, \m+n_3)\setminus\D(\m+n_1,\m+n_3))\cap [0,\m)\subseteq B$. Thus $\cardinal ((\D(\m+n_1+(h_2+1)a, \m+n_3)\setminus\D(\m+n_1,\m+n_3))\cap [0,\m))\le \cardinal (B\cap [0,\m))\le \cardinal (A\cap[0,\m))=\cardinal ((\D(\m+n_1,\m+n_2,\m+n_3)\setminus\D(\m+n_1,\m+n_3))\cap [0,\m))$.
\end{proof}

\begin{remark}\label{vale-para-dos}
  Lemmas~\ref{de-mas-enmedio} to~\ref{carga-cambio} also hold if we only take $n_1, n_2\in \mathbb N$ with $\m\in L_1\prec L_2$. The role played by $v_3$ is in this setting played by $v_1-a=-\peb[i_1a/b]$ if $i_1\neq 0$ and by $-a$ otherwise.
\end{remark}

We are going to prove that in addition to the condition $\m\in M$ (Proposition~\ref{condiciones-ms}), in order to find an optimal configuration, we can also assume that the set $M\cap[\m,\m+b)$ is an
amenable interval.

\begin{lemma}\label{puedo-tomar-intervalo}
  Let $M\subseteq[\m,\infty)$ be an amenable set with $\m\in M$, such that $\cardinal \D(M)$ is the minimum of $\cardinal \D(M')$ with $M'\subseteq [\m,\infty)$ amenable, $\m\in M'$ and $\cardinal M=\cardinal M'$. Then we can assume that $M\cap[\m,\m+b)$ is an amenable interval.
\end{lemma}
\begin{proof}
  We know that $L=M\cap[\m,\m+b)$ is of the form $L=L_1\cup \cdots \cup L_t$ with $L_1,\ldots, L_t$ amenable intervals such that $\m\in L_1\prec \cdots \prec L_t$. Take $M$ with $t$ minimum. Assume that $t>1$.

  Let $r=\cardinal M$. Let $n_1,\ldots, n_t\in \{0,\ldots, ab-1\}$ be such that $\D(\m+n_i)\cap[\m,\m+b)=L_i$, and let  $i_j,h_j\in\{0,\ldots, b-1\}$ be such that $n_j=i_ja\bmod b+ h_ja$, $j\in \{1,\ldots, t\}$ (Remark~\ref{hat-en-S}). Let $D= \D(\m+n_1,\ldots,\m+n_t) \cap[\m,\infty)$. Then $D\cap[\m,\m+b)=L$, and by Lemma~\ref{divisibilidad-y-suelos}, $M\subseteq D$ and $D$ is an amenable set.

  Consider now $D'=\D(\m+n_1+(h_2+1)a,\m+n_3,\ldots,\m+n_t)\cap [\m,\infty)$. Then $\cardinal D'\ge r$ and $\cardinal (\D(D')\cap [0,\m))\le \cardinal (\D(D)\cap [0,\m))=\cardinal (\D(M)\cap [0,\m))$ in view of Lemmas~\ref{capacidad-cambio} and~\ref{carga-cambio}, Propositions~\ref{solo-cuentan-los-contiguos} and~\ref{sombra}, and Remark~\ref{vale-para-dos}. Observe that if we set $L'=D'\cap [\m,\m+b)$, then $\cardinal L=\cardinal L'$ by Lemma~\ref{desc-suelo-D}.

  Finally we construct $M'$ by changing $D'$ with $D'\setminus\{\max(D')\}$ as many times as needed until $M'$ has $r$ elements. We can do this because $\cardinal D'\ge r$. Then $M'$ is amenable and $M'\cap [\m,\m+b)=L'$ (this last assertion holds because $\cardinal L=\cardinal L'$, and thus in the process of removing $\max(D')$ we never take elements in $L'$). By Proposition~\ref{sombra}, $\cardinal \D(M')=\cardinal (\D(D')\cap [0,\m)) + r\le \cardinal (\D(D)\cap [0,\m))+r=\cardinal (\D(M)\cap[0,\m))+r= \cardinal \D(M)$. The minimality of $\cardinal \D(M)$ forces $\cardinal \D(M')=\cardinal \D(M)$. However, $L'$, in its decomposition as amenable intervals, has one interval less than $L$, contradicting the minimality of $t$.
\end{proof}

Our next goal is to prove that the amenable set $\D(\m +\rho_r)\cap [m,\infty)$ is an optimal configuration (actually with $r$ elements in light of Remark \ref{orden-triangulos}). First we need a result comparing the divisors below $\m$ while we move upwards in $S$.

\begin{lemma}\label{rho-i-j}
  For every $t,r\in \mathbb N$, with $t\ge r$, $\cardinal (\D(\m+\rho_t)\cap[0,\m))\ge \cardinal (\D(\m+\rho_r)\cap [0,\m))$.
\end{lemma}
\begin{proof}
  Observe that $\rho_t-\rho_{t-1}\ge 1$. Hence by induction $\rho_t-\rho_r\ge t-r$.

  From Proposition~\ref{prop:well-known}, we deduce that $\cardinal \D(\m+\rho_t)=\cardinal \D(\m)+\rho_t$ and $\cardinal \D(\m+\rho_r)=\cardinal \D(\m)+\rho_r$. Hence, $\cardinal \D(\m+\rho_r)+\rho_t-\rho_r=\cardinal \D(\m+\rho_t)$.

  By Proposition~\ref{sombra} and Remark~\ref{orden-triangulos}, $\cardinal \D(\m+\rho_t)=\cardinal (\D(\m+\rho_t)\cap[0,\m))+t$ and $\cardinal \D(\m+\rho_r)=\cardinal (\D(\m+\rho_r)\cap[0,\m))+r$. Hence, $\cardinal (\D(\m+\rho_t)\cap[0,\m))= \cardinal (\D(\m+\rho_r)\cap [0,\m))+\rho_t-\rho_r-(t-r)$. As $\rho_t-\rho_r\ge t-r$ we conclude that $\cardinal (\D(\m+\rho_t)\cap[0,\m))\ge \cardinal (\D(\m+\rho_r)\cap [0,\m))$.
\end{proof}

\begin{lemma}\label{puedo-tomar-div-n}
  Let $M\subseteq[\m,\infty)$ be an amenable set with $\m\in M$, and such that $\cardinal \D(M)$ is the minimum of $\cardinal \D(M')$ with $M'\subseteq [\m,\infty)$ amenable, $\m\in M'$ and $\cardinal M=\cardinal M'$. Then $\cardinal \D(M)=\m+1-2g+\rho_r$, where $r=\cardinal M$.
\end{lemma}
\begin{proof}
  In light of Lemma~\ref{puedo-tomar-intervalo}, we may assume that $L=M\cap[\m,\m+b)$ is an amenable interval. It may happen that $L$ coincides with the ground or that it is strictly contained in it. We consider these two cases separately.

\begin{enumerate}[1.]
  \item $L=\{\m,\ldots, \m+b-1\}$. Let $L'=\D(\m+\rho_r)\cap[\m,\m+b)$. In view of Proposition~\ref{sombra}, $\cardinal \D(M)= \cardinal (\D(M)\cap [0,\m))+r$. Also, $\cardinal (\D(M)\cap [0,\m))=\cardinal(\D(L)\cap [0,\m))\ge \cardinal(\D(L')\cap[0,\m))=\cardinal(\D(\m+\rho_r)\cap[0,\m)$. Hence $\cardinal\D(M)\ge \cardinal(\D(\m+\rho_r)\cap[0,\m))+r=\cardinal (\D(\m+\rho_r)\cap[0,\m))+\cardinal(\D(\m+\rho_r)\cap[\m,\infty)$ (Remark~\ref{orden-triangulos}). We conclude that $\cardinal\D(M)\ge \cardinal\D(\m+\rho_r)$, and, by minimality of $\cardinal \D(M)$, the equality holds. Proposition~\ref{prop:well-known} then asserts that $\cardinal \D(M)=\m+1-2g+\rho_r$.
  \item $L\neq\{\m,\ldots, \m+b-1\}$. Let $n\in\mathbb N$ be such that $L=\D(\m+n)\cap[\m,\m+b)$. Such an element exists by Lemma~\ref{desc-suelo-D}. Since $\m\in L\subseteq \D(\m+n)$, we have that $n=\m+n-\m\in S$. Let $D=\D(\m+n)\cap[\m,\infty)$.

  Let $t=\cardinal D$. By Remark~\ref{orden-triangulos}, we have that $n=\rho_t$. In view of Corollary~\ref{cabeza-es-triangulo}, $r\le t$. From Lemma~\ref{rho-i-j} follows that $\cardinal (\D(\m+\rho_t)\cap[0,\m))\ge \cardinal (\D(\m+\rho_r)\cap [0,\m))$.

  Proposition~\ref{sombra} ensures that $\cardinal \D(M)= \cardinal (\D(L)\cap[0,\m))+r= \cardinal (\D(\m+\rho_t)\cap[0,\m))+r \ge \cardinal (\D(\m+\rho_t)\cap [0,\m))+r = \cardinal \D(\m+\rho_t)$. By the minimality of $\cardinal \D(M)$, we get $\cardinal \D(M)=\cardinal \D(\m+\rho_r)$.

  Finally, it suffices to use the equality $\cardinal \D(\m+\rho_r)=\m+\rho_r+1-2g$ (Proposition~\ref{prop:well-known}). \qedhere
\end{enumerate}
\end{proof}

Now that we know that $\D(\m +\rho_r)\cap [m,\infty)$ is an optimal configuration with $r$ elements, computing $\E(S,r)$ is an easy task.

\begin{theorem}\label{main}
  Let $S=\{0=\rho_1<\rho_2<\cdots <\rho_n<\cdots\}$ be an embedding dimension two numerical semigroup. Then $\mathrm E(S,r)=\rho_r$.
\end{theorem}
\begin{proof}
  Follows from the definition of $\E(S,r)$, Proposition~\ref{condiciones-ms} and Lemma~\ref{puedo-tomar-div-n}.
\end{proof}

Since embedding dimension two numerical semigroups are symmetric, by using the fact that $\delta_{FR}^{r}(m)\geq m+1-2g+\mathrm E(S,r)$ for $m\geq c$, and equality holds if $m=2g-1+\rho_{k}$ for some $k\geq 2$, one easily obtains the following consequence.

\begin{corollary}\label{coro-final}

Let $S=\{0=\rho_1<\rho_2<\cdots <\rho_n<\cdots\}$ be an embedding dimension two numerical semigroup. Then

\begin{enumerate}[1.]

\item $\delta_{FR}^{r}(m)=\rho_{r}+\rho_{k}$ if $m=2g-1+\rho_{k}$ with $k\geq 2$,

\item $\delta_{FR}^{r}(m)\geq \rho_{r}+\ell_{i}$ if $m=2g-1+\ell_{i}$, where $\ell_{i}\in \mathrm G(S)$ is a gap of $S$, for $i\in\{1,\ldots,g\}$.

\end{enumerate}

\end{corollary}

\begin{remark}
The above result, together with \cite[Theorem 5.5]{KirPel} suggests the question of whether the following formula holds for $m\geq c$ in a numerical semigroup generated by two elements:
\[
\delta_{FR}^{r}(m)=\min\{\rho_{r}+\rho_{k}\;|\;\rho_{k}\geq m+1-2g\}.
\]
However, this question has in general a negative answer. In fact, consider the semigroup $S=\langle 2,5\rangle$ (hyperelliptic) with genus $g=2$ and conductor $c=4$. If we take $r=3$ and $m=4=(2g-1)+1$, the Feng-Rao number is $E_{3}=\rho_{3}=4$, so that the Feng-Rao distance is
\[
\delta_{FR}^{3}(4)\geq 5,
\]
and the result of applying the above formula is 6.  Nevertheless, the Feng-Rao distance is actually $\delta_{FR}^{3}(4)=5$, since
\[
\D(4,5,7) = \{ 0, 2, 4, 5, 7 \}.
\]
\end{remark}

\section{Examples and conclusions}\label{sec:examples_conclusions}

The results of the previous section, in particular Corollary \ref{coro-final}, allows easily to prove the following Theorem \ref{thm-final}, improving the Theorem 2.8 in~\cite{KirPel}. We first recall the definition of the generalized (Hamming) weights.  In fact, we define the support of a linear code $C$ as
\[ {\rm supp}(C):=\{i\,|\,c_{i}\neq 0\;\;\mbox{for some ${\bf c}\in C$}\}.
\]
Thus, the $r$th generalized weight of $C$ is defined by
\[
{\mathrm d}_{r}(C):=\min\{\sharp\,{\rm supp}(C')\;|\;\mbox{$C'$ is a linear subcode of $C$ with ${\rm dim}(C')=r$}\}.
\]
Of course, the above definition only makes sense if $r\leq k$, where $k$ is the dimension of $C$.  The set of numbers
\[
\mathrm{GHW}(C):=\{{\mathrm d}_{1},\ldots,{\mathrm d}_{k}\}
\]
is called the {\em weight hierarchy} of the code $C$ (see~\cite{CMunH}).

\begin{theorem}\label{thm-final}

  Let $S=\{0=\rho_1<\rho_2<\cdots <\rho_n<\cdots\}$ be an embedding dimension two numerical semigroup. Then
  \[
  {\mathrm d}_{r}(C_{m})\geq\delta_{FR}(m+1)+\rho_{r}
  \]
  for $r=1,\ldots,k_{m}$, where $C_{m}$ is a code in an array of codes as in~\cite{KirPel} (for example, $C_{m}$ being a one-point AG code associated to a divisor of the form $G=mP$), and $k_{m}$ is the dimension of $C_{m}$.
\end{theorem}

\begin{proof}
  Since $E(S,r)=\rho_{r}$ and $\delta_{FR}^{r}(m)\geq m+1-2g+\mathrm E(S,r)$ for $m\geq c$, we just apply that ${\mathrm d}_{r}(C_{m})\geq\delta_{FR}^{r}(m+1)$.
\end{proof}

Note that $k_{m}$ depends not only on $m$, but also on the length of $C_{m}$ in the array of codes.  For example, if $C_{m}\equiv C_{\Omega}(D,mP)$ is again a one-point AG code, it depends on the number of
points $n$ that are used for evaluation, that is
\[
k_{m}=n-\sharp\{\rho\in S\,|\,\rho\leq m\}=n-m+g-1,
\]
provided $2g-2<m<n$ and $m\in S$.

\begin{remark}

  Theorem \ref{thm-final} improves \cite[Theorem 2.8]{KirPel}, which states
  \[
  {\mathrm d}_{r}(C_{m})\geq\delta_{FR}(m+1)+(r-1).
  \]
  This inequality is actually a consequence of the inequality $\delta_{FR}^{r}(m)\geq m+1-2g+\E(S,r)$, by taking into account that $\E(S,r)\geq r-1$. In fact $\E(S,r)\geq r$ if the genus of $S$ is $g>0$ (see~\cite{WCC}). The improvement follows from the fact that $\E(S,r)=\rho_{r}$ is larger than $r-1$ if $r\geq 2$ and $g\geq 1$.

\end{remark}

On the other hand, the generalized Griesmer bound for the generalized Hamming weights states that
\[
{\mathrm d}_{r}(C)\geq\displaystyle\sum_{i=1}^{r-1}\left\lceil\displaystyle\frac{\mathrm d(C)}{q^{i}}\right\rceil,
\]
where $\mathrm d(C)\equiv {\mathrm d}_{1}(C)$ is the minimum distance of the code $C$, which is defined over the finite field $\mathbb{F}_{q}$ (see~\cite{Griesmer}).  In particular, for $r=2$ one has
\[
{\mathrm d}_{2}(C)\geq \mathrm d(C)+\left\lceil\displaystyle\frac{\mathrm d(C)}{q}\right\rceil.
\]
Since we are just using the semigroup for estimating the generalized Hamming weights, we can substitute $\mathrm d(C_{m})$, $C_{m}$ being in an array of codes as in~\cite{KirPel}, by the order bound $\delta_{FR}(m+1)$ obtaining the bound
\[
{\mathrm d}_{r}(C_{m})\geq\displaystyle\sum_{i=1}^{r-1}\left\lceil\displaystyle\frac{\delta_{FR}(m+1)}{q^{i}}\right\rceil.
\]
We may call this bound the {\em Griesmer order bound}. For $r=2$ this bound becomes
\[
{\mathrm d}_{2}(C)\geq\delta_{FR}(m+1)+\left\lceil\displaystyle\frac{\delta_{FR}(m+1)}{q}\right\rceil.
\]
The maximum values of these bounds are achieved in the binary case $q=2$.

\begin{remark}

We have previously remarked that our bound in Theorem \ref{thm-final} is better than the one in~\cite{KirPel}.  The difference of both bounds is constant in $m\geq c$, when $r$ is fixed. 

In order to compare the bound in Theorem \ref{thm-final} with the Griesmer order bound, we first note that such a comparison depends on several parameters, namely the cardinality $q$ of the finite field, the order $r$ and the element $m$ in the semigroup. Here we present some conclusions from our experimental results. 

\begin{enumerate}[$\bullet$]

\item We first note that in the following tables there will be a delay of one unit, because the bound for the code $C_{m}$ corresponds to $m+1$ in the Feng-Rao distances. More precisely, in the first row of the tables $m$ corresponds to the code $C_{m}$, whereas the Feng-Rao distances used in the second and third rows correspond to $m+1$. 

On the other hand, note that for a semigroup generated by two elements, the minimum formula 
\begin{equation}\label{min-formula}
\delta_{FR}(m+1)=\min\{\rho_{k}\;|\;\rho_{k}\geq m+2-2g\}
\end{equation}
holds for $m\geq c$ (see~\cite{KirPel}). Thus, the classical Feng-Rao distance comes in bursts of repeated values, according to intervals of gaps (deserts) of the form $m+2-2g$ preceding the $\rho_{k}$ achieving the minimum in Formula \eqref{min-formula}. As a consequence, the corresponding Griesmer order bound also comes in bursts, and jumps just after the corresponding $\rho_{k}$. 

\item Our bound is increasing one by one with $m$, while the Griesmer order bound jumps at values of $m$ corresponding to gaps of the form $m+2-2g$ starting a desert. Moreover, when there is no such a gap, the Griesmer order bound increases by one or more. Therefore,  this bound tends to improve our bound as $m$ becomes large, or if $m$ corresponds to a gap at the beginning of a long desert. Nevertheless, our bound seems to be better for small values of $m$ of the form $2g-2+\rho_{k}$, and also at the end of the desert preceding to such a $\rho_{k}$. For example, for $S=\langle7,11\rangle$ and $r=2$ we obtain with \GAP\  the following results for $q=2$,
    \begin{center}
      \begin{tabular}{|c|cccccccccccccccccc|}
        \hline 
        m&30&31&32&33&34&35&36&37&38&39&40&41&42&43&44&45&46&$\cdots$\\ 
        \hline 
        GFR&9&10&11&{\bf 12}&{\bf 13}&{\bf 14}&15&16&17&{\bf 18}&19&20&21&22&23&24&25&$\cdots$\\ 
        \hline 
        GOB&{\bf 11}&{\bf 11}&11&11&11&11&{\bf 17}&{\bf 17}&17&17&{\bf 21}&{\bf 21}&21&{\bf 27}&{\bf 27}&{\bf 27}&{\bf 27}&$\cdots$\\ 
        \hline 
      \end{tabular}
    \end{center}
    where the row GFR corresponds to our bound with the generalized Feng-Rao distance, and the row GOB corresponds to the Griesmer order bound.

\item On the other hand, if we increase $r$ the difference of both bounds for $m=c$ becomes larger, so that our bound GFR remains better for more values of $m$.  For instance, if we take $r=10$ in the previous example, GFR is better for $m\leq 50$ and GOB is better for $m\geq 51$: 
    \begin{center}
      \begin{tabular}{|c|cccccccccccccccc|}
        \hline 
        m&30&31&32&33&34&35&36&37&38&39&40&$\cdots$&50&51&52&53\\ 
        \hline 
        GFR&31&32&33&34&35&36&37&38&39&40&41&$\cdots$&51&52&53&54\\ 
        \hline 
        GOB&20&20&20&20&20&20&28&28&28&28&33&$\cdots$&49&56&56&56\\ 
        \hline 
      \end{tabular}
    \end{center}

\item Finally, as the size $q$ of the finite field increases, the jumps of the Griesmer order bound become smaller, so that our bound is better for much more values of $m$. For example, if we switch in the last example to $q=16$, our bound is much better in the whole interval $c\leq m\leq 2c-1$. 

\item In general, the experimental results above suggest that a good strategy to estimate the generalized Hamming weights by means of the underlying numerical semigroup $S$ is to combine both, the generalized Feng-Rao distances and the Griesmer order bound, depending on the parameters $q$, $r$ and $m$. Roughly speaking, our bound GFR is better for $q$ and $r$ large, whereas the bound GOB is better otherwise, provided $m$ is large or it corresponds to a gap of the form $m+1-2g$ at the beginning of a long desert. 

\end{enumerate}

\end{remark}

We finally test these bounds in the case of Hermitian codes.

\begin{example}

Consider the Hermitian codes over $\mathbb{F}_{16}$ (see~\cite{HvLP} for further details).  The involved semigroup is $S=\langle 4,5\rangle$ and the length of the codes is $n=64$.  Since the conductor is $c=12$ and the genus is $g=6$, the dimension of the codes is $69-m$, for $12\leq m\leq63$. Our computations with GAP show that our bound GFR is always better than (or equal to) the Griesmer order bound. For example, if $r=2$ we obtain 
\begin{center}
    \begin{tabular}{|c|cccccccccccccccccccccc|}
      \hline 
      m&12&13&14&15&16&17&18&19&20&21&22&23&24&25&26&27&28&$\cdots$&58&59&$\cdots$&63\\ 
      \hline 
      GFR&6&7&8&9&10&11&12&13&14&15&16&17&18&19&20&21&22&$\cdots$&52&53&$\cdots$&57\\ 
      \hline 
      GOB&5&5&5&6&9&9&9&10&11&13&13&14&15&16&17&19&20&$\cdots$&51&53&$\cdots$&57\\ 
      \hline 
    \end{tabular}
\end{center}
and our bound GFR improves as $r$ gets higher. In fact, for $r\geq 4$ the GFR bound is strictly better than the Griesmer one. 
\end{example}


\begin{thebibliography}{100}

\bibitem{AngCar} A. Barbero and C. Munuera, \lq\lq The weight hierarchy of
  Hermitian codes\rq\rq, {\em SIAM J. Discrete Math.} \textbf{13}, no. 1,
  79-104 (2000).

\bibitem{WSPink} A. Campillo and J.I. Farr\'{a}n, \lq\lq Computing
  Weierstrass semigroups and the Feng-Rao distance from singular plane
  models\rq\rq, {\em Finite Fields and their Applications} {\bf 6}, 71-92
  (2000).

\bibitem{Arf} A. Campillo, J.I. Farr\'{a}n and C. Munuera, \lq\lq On the
  parameters of algebraic geometry codes related to Arf semigroups\rq\rq,
  {\em IEEE Trans. of Information Theory} {\bf 46}, 2634-2638 (2000).

\bibitem{Singular} W. Decker, G.-M. Greuel, G. Pfister and H. Sch\"{o}nemann,
  \lq\lq {\sc Singular} 3-1-3", {\em a computer algebra system for polynomial
    computations}, Centre for Computer Algebra, University of
  Kaiserslautern (2011).  Available via {\tt http://www.singular.uni-kl.de/}.

\bibitem{intpic} M.~Delgado.  \newblock ``IntPic'', \emph{a
    \textsf{GAP} package for drawing integers}, Available via\\
  \texttt{http://www.fc.up.pt/cmup/mdelgado/software/}.

\bibitem{numericalsgps} M. Delgado, P. A. Garc\'{\i}a-S\'{a}nchez and
  J. Morais, \lq\lq NumericalSgps\rq\rq, \emph{A \GAP\ package for numerical
    semigroups}.  Available via {\tt http://www.gap-system.org/}.

\bibitem{fr-intervalos} M. Delgado, J. I. Farr\'{a}n,
  P. A. Garc\'{\i}a-S\'{a}nchez and D. Llena, \lq\lq On the generalized
  Feng-Rao numbers of numerical semigroups generated by intervals\rq\rq,
  Math. Comput. \textbf{82} (2013), 1813-1836.

\bibitem{JMDA} J. I. Farr\'{a}n, P. A. Garc\'{\i}a-S\'{a}nchez and D. Llena,
  \lq\lq On the Feng-Rao numbers\rq\rq, {\em Actas de las VII Jornadas de
    Matem\'{a}tica Discreta y Algor\'{\i}tmica}, 321-333 (2010).

\bibitem{brnoeth} J.I. Farr\'{a}n and Ch. Lossen, \lq\lq {\tt brnoeth.lib}",
  {\em A} {\sc SINGULAR} 2.0 {\em library for the Brill-Noether algorithm,
    Weierstrass semigroups and AG codes} (2001).  Available via {\tt
    http://www.singular.uni-kl.de/}.

\bibitem{WCC} J. I. Farr\'{a}n and C. Munuera, \lq\lq Goppa-like bounds for
  the generalized Feng-Rao distances\rq\rq, {\em Discrete Applied
    Mathematics} {\bf 128}/1, 145-156 (2003).

\bibitem{FR} G.L. Feng and T.R.N. Rao, \lq\lq Decoding algebraic-geometric
  codes up to the designed minimum distance\rq\rq, {\em IEEE
    Trans. Inform. Theo\-ry} {\bf 39}, 37-45 (1993).

\bibitem{GAP4} The GAP~Group, \emph{GAP -- Groups, Algorithms, and
    Programming, Version 4.5}; 2012, \newblock Available via
  \texttt{http://www.gap-system.org/}.

\bibitem{HeiPel} P. Heijnen and R. Pellikaan, \lq\lq Generalized Hamming
  weights of $q$-ary Reed-Muller codes\rq\rq, {\em IEEE
    Trans. Inform. Theo\-ry} {\bf 44}, 181-197 (1998).

\bibitem{HKM} T. Helleseth, T. Kl\o ve and J. Mykkleveit, \lq\lq The weight
  distribution of irreducible cyclic codes with block lengths
  $n_{1}((q^{l}-1)/N)$\rq\rq, {\em Discrete Math.}, \textbf{18}, 179-211
  (1977).

\bibitem{Griesmer} T. Helleseth, T. Kl\o ve and \O. Ytrehus: \lq\lq
  Generalizations of the Griesmer bound\rq\rq, in \emph{Error Control,
    Cryptology, and Speech Compression}, LNCS {\bf 829}, pp. 41-52,
  Springer (1994).

\bibitem{HvLP} T. H\o holdt, J.H. van Lint and R. Pellikaan, \lq\lq Algebraic
  Geometry codes\rq\rq, in \emph{Handbook of Coding Theory}, V. Pless,
  W.C. Huffman and R.A. Brualdi, Eds., 871-961 (vol. 1), Elsevier, Amsterdam
  (1998).

\bibitem{KirPel} C. Kirfel and R. Pellikaan, \lq\lq The minimum distance of
  codes in an array coming from telescopic semigroups\rq\rq, {\em IEEE
    Trans. Inform. Theory} {\bf 41}, 1720-1732 (1995).

\bibitem{Komeda} J. Komeda, \lq\lq On the existence of Weierstrass points
  with a certain semigroup generated by 4 elements\rq\rq, {\em Tsukuba
    J. Math.} Vol. 6 No. 2, pp. 237-270 (1982).

\bibitem{CMunH} C Munuera: \lq\lq Generalized Hamming Weights and Trellis
  Complexity\rq\rq, in \emph{Advances in Algebraic Geometry Codes},
  E. Mart\'{\i}nez-Moro, C. Munuera, D. Ruano (Eds.), pp. 363-389, World
  Scientific (2008).

\bibitem{NS} J. C. Rosales and P. A. Garc\'{\i}a-S\'{a}nchez, \lq\lq
  Numerical Semigroups\rq\rq, {\em Developments in Maths.} \textbf{20},
  Springer (2010).

\bibitem{Stichtenoth} H. Stichtenoth, \lq\lq Algebraic Function Fields and
  Codes (Second Edition)\rq\rq, {\em Graduate Texts in Mathematics} {\bf
    254}, Springer-Verlag (2009).

\bibitem{Wei} V. Wei, \lq\lq Generalized Hamming weights for linear
  codes\rq\rq, {\em IEEE Trans. Inform. Theory} {\bf 37}, 1412-1428 (1991).

\end{thebibliography}
\end{document}